\documentclass[a4paper]{amsart}

\usepackage{amssymb}

\newtheorem{theorem}[equation]{Theorem}
\newtheorem{lemma}[equation]{Lemma}
\newtheorem{corollary}[equation]{Corollary}
\newtheorem{proposition}[equation]{Proposition}

\numberwithin{equation}{section}

\theoremstyle{definition}

\newtheorem*{example*}{Example}

\newtheorem{remark}[equation]{Remark}
\newtheorem*{remark*}{Remark}

\newcommand{\bA}{{\mathbb A}}
\newcommand{\bZ}{{\mathbb Z}}

\newcommand{\frg}{{\mathfrak g}}

\newcommand{\frf}{{\mathfrak f}}

\newcommand{\frt}{{\mathfrak t}}
\newcommand{\frs}{{\mathfrak s}}

\newcommand{\calB}{{\mathcal B}}

\newcommand{\calO}{{\mathcal O}}

\newcommand{\subo}{_{\bar 0}}
\newcommand{\subuno}{_{\bar 1}}

\DeclareMathOperator{\eespan}{span}
\providecommand{\espan}[1]{\eespan\left\{ #1\right\}}

 \newcommand{\tri}{\mathfrak{tri}}

 \newcommand{\fro}{\mathfrak{o}}
 \newcommand{\frsl}{{\mathfrak{sl}}}

 \DeclareMathOperator{\tr}{tr}
  
 \DeclareMathOperator{\ad}{ad}
 \DeclareMathOperator{\der}{\mathfrak{der}}

 \DeclareMathOperator{\End}{End}
 
 \DeclareMathOperator{\Mat}{Mat}

 \DeclareMathOperator{\aalg}{alg}
\providecommand{\alg}[1]{\aalg\left\langle #1\right\rangle}

\def\bigstrut{\vrule height 14pt width 0ptdepth 2pt}

\def\hregleta{\hrule height .5pt}
\def\hreglon{\hrule height1pt}
\def\vreglon{\vrule height 14pt width1pt depth 4pt}
\def\vregleta{\vrule width .5pt}
\def\hreglonfill{\leaders\hreglon\hfill}

\def\hregletafill{\leaders\hregleta\hfill}

\newenvironment{romanenumerate}
 {\begin{enumerate}
 
 }{\end{enumerate}}

\begin{document}

\title{Gradings on Symmetric Composition Algebras}

\author[Alberto Elduque]{Alberto Elduque$^{\star}$}
 \thanks{$^{\star}$ Supported by the Spanish Ministerio de
 Educaci\'{o}n y Ciencia
 and FEDER (MTM 2007-67884-C04-02) and by the
Diputaci\'on General de Arag\'on (Grupo de Investigaci\'on de
\'Algebra)}
 \address{Departamento de Matem\'aticas e
 Instituto Universitario de Matem\'aticas y Aplicaciones,
 Universidad de Zaragoza, 50009 Zaragoza, Spain}
 \email{elduque@unizar.es}


\date{\today}



\begin{abstract}
The group gradings on the symmetric composition algebras over arbitrary fields are classified. Applications of this result to gradings on the exceptional simple Lie algebras are considered too.
\end{abstract}

\maketitle


\section{Introduction}\label{se:Introduction}

The gradings over groups of the simple classical Lie algebras (other than $D_4$) over an algebraically closed field of characteristic $0$ were considered in \cite{PateraetalI,PateraetalII}. A description of all the possibilities, as well as of gradings on simple Jordan algebras, appears in \cite{BahturinShestakovZaicev}, based on previous results on gradings on the simple associative algebras (that is, on matrix algebras) in \cite{Bahturinetal}.

On the other hand, the gradings on octonion algebras were classified over arbitrary fields in \cite{GradingsOctonions}.

Quite recently, there has been a renewed interest in gradings on exceptional simple Lie algebras over algebraically closed fields of characteristic $0$. Not surprisingly, the gradings on $G_2$ (\cite{CristinaGradingsG2,BahturinTvalavadze}) are strongly related to the gradings on the octonions. Gradins on $D_4$ have been considered in \cite{DMV}, while the gradings on $F_4$ and on the exceptional simple Jordan algebra (or Albert algebra) have been classified in \cite{DraperMartinF4}. In these latter papers, still the gradings on the octonions play an important role, but new possibilities appear.

Some of these new possibilities can be explained in terms of gradings of the so called symmetric composition algebras. These form a class of not necessarily unital composition algebras with quite nice properties (see \cite[Chapter 8]{KMRT} and the references there in).

\smallskip

The aim of this paper is the complete classification of the group gradings on symmetric composition algebras over arbitrary fields. Applications of these gradings to the description of some interesting fine gradings on exceptional simple Lie algebras will be given too.

\medskip

Let us first review the basic definitions on gradings that will be used throughout the paper.

Following \cite{PateraetalI}, a grading on a (not necessarily associative) algebra $A$ over a field $k$ is a decomposition
\begin{equation}\label{eq:AoplusAg}
A=\oplus_{g\in G}A_g,
\end{equation}
into a direct sum of subspaces $A_g$, $g\in G$, such that for any two indices $g_1,g_2\in G$, either $A_{g_1}A_{g_2}=0$ or there exists an element $g_3\in G$ such that $0\ne A_{g_1}A_{g_2}\subseteq A_{g_3}$.

Gradings give a very useful tool to study complicated objects by splitting them into nicer smaller components.

\smallskip

Given another grading
\begin{equation}\label{eq:AoplusAh}
A=\oplus_{h\in H}A_h,
\end{equation}
of $A$, the grading in \eqref{eq:AoplusAg} is said to be a \emph{coarsening} of the one in \eqref{eq:AoplusAh}, and then this latter one is called a \emph{refinement} of the former, in case for any $h\in H$ there is a $g\in G$ with $A_h\subseteq A_g$. In other words, for any $h\in H$, the subspace $A_h$ is a direct sum of subspaces $A_g$.
A \emph{fine} grading is a grading which admits no proper refinement.

Two gradings \eqref{eq:AoplusAg} and \eqref{eq:AoplusAh} are said to be \emph{equivalent} if there is an automorphism $\varphi$ of $A$ such that for any $g\in G$ with $A_g\ne 0$, there is an $h\in H$ with $\varphi(A_g)=A_h$.

\medskip

The most interesting gradings are those for which the index set $G$ is a group and for any $g_1,g_2\in G$, $A_{g_1}A_{g_2}\subseteq A_{g_1g_2}$. These are called \emph{group gradings}. This is not always the case. See \cite{Elduquenosemigroup} for an example of a grading on a Lie algebra which is not even a grading over a semigroup. Also, a grading of the split Cayley algebra will be given in Remark \ref{re:nongroupgrading}, which is not a group grading.

Given a group grading $A=\oplus_{g\in G}A_g$, it will be always assumed that the group $G$ is generated by its subset $\{g\in G:A_g\ne 0\}$.

It is proved in \cite{GradingsOctonions} that the grading group of any Cayley algebra is always abelian.

\smallskip

In order to avoid equivalent gradings, given any grading $A=\oplus_{g\in G}A_g$ of an arbitrary algebra, we will consider the \emph{universal grading group}, which is defined (see \cite{Elduquenosemigroup} or \cite{CristinaGradingsG2}) as the quotient $\hat G=\bZ (G)/R$ of the abelian group $\bZ(G)$ freely generated by the set $G$, modulo the subgroup $R$ generated by the set $\{a+b-c: a,b,c\in G,\ 0\ne A_aA_b\subseteq A_c\}$. Then $A$ is $\hat G$-graded with $A_\gamma=\sum\{A_g:g+R=\gamma\}$.

It is clear that if the given grading $A=\oplus_{g\in G}A_g$ is already a group grading with abelian $G$, then $G$ is a quotient of the universal grading group $\hat G$ and the given grading is equivalent to the new grading $A=\oplus_{\gamma\in\hat G}A_\gamma$ (here the automorphism $\varphi$ can be taken to be the identity). Therefore, in dealing with gradings over abelian groups, up to equivalence, it is enough to consider the universal grading groups.

\medskip

The paper is structured as follows:

Section \ref{se:Compo} will review some properties of the unital composition algebras. The classification of the gradings on octonion algebras in \cite{GradingsOctonions} will be reviewed in a way suitable for our purposes. Section \ref{se:SymCompo} will introduce the symmetric composition algebras. Their main features will be recalled and even some new results proved. Roughly speaking, these algebras split into two classes: para-Hurwitz algebras and Okubo algebras. But while para-Hurwitz algebras inherit the gradings of their Hurwitz (that is, classical unital composition algebras) counterparts, Okubo algebras present natural gradings over $\bZ_3^2$, which are not present in the para-Hurwitz case. Besides, the split Okubo algebra (that is, the so called \emph{pseudo-octonion} algebra) can be presented in several different ways, each one giving rise to a grading.

Section \ref{se:GradSymCompo} is the core of the paper. Here a complete classification of the group gradings of the symmetric composition algebras is given in Theorem \ref{th:gradingssymmetric}. It turns out that over an algebraically closed field of characteristic $\ne 3$, any fine grading of the pseudo-octonion algebra is either a $\bZ^2$-grading, which comes from the weight space decomposition relative to a Cartan subalgebra of its Lie algebra of derivations, or the $\bZ_3^2$-grading mentioned above.

Finally, Section \ref{se:GradLie} is devoted to show how the different gradings on symmetric composition algebras, together with the construction of the exceptional simple Lie algebras in terms of two such composition algebras given in \cite{Eld04Ibero}, can be combined to obtain fine gradings on the latter algebras. In particular, two interesting $\bZ_3^5$ and $\bZ_2^8$-gradings of $E_8$ will be obtained. A natural coarsening of the $\bZ_2^8$-grading gives a $\bZ_2^5$-grading $\frg=\oplus_{0\ne a\in\bZ_2^5}\frg_a$, where $\frg_a$ is a Cartan subalgebra of $\frg$ for any $a$, thus obtaining the Dempwolff decomposition considered in \cite{Thompson}.

\bigskip

\section{Composition algebras}\label{se:Compo}

This section will be devoted to review some known facts and features of composition algebras. For details one may consult \cite[Chapter 8]{KMRT} or \cite[Chapter 2]{ZSSS}.

Composition algebras with a unity element constitute a well-known class of algebras. They generalize the classical algebras of the reals, complex, quaternions and octonions.

A triple $(C,\cdot,n)$ consisting of a vector space $C$ over a ground field $k$, endowed with a bilinear multiplication $C\times C\rightarrow k$, $(x,y)\mapsto x\cdot y$, and a nondegenerate quadratic form $n:C\rightarrow k$ permitting composition, that is $n(x\cdot y)=n(x)n(y)$, is called a \emph{composition algebra}. Here the norm being nondegenerate will mean that its associated polar form defined by $n(x,y)=n(x+y)-n(x)-n(y)$ is nondegenerate: $\{x\in C: n(x,y)=0\ \textrm{for any $y\in C$}\}=0$.

For simplicity, sometimes we will refer simply to the composition algebra $C$, if the underlying multiplication and norm are clear from the context.

\smallskip

The unital composition algebras, also termed \emph{Hurwitz algebras}, form a class of degree two algebras, as any element satisfies the Cayley-Hamilton equation:
\begin{equation}\label{eq:CayleyHamilton}
x^{\cdot 2}-n(x,1)x+n(x)1=0
\end{equation}
for any $x$. Besides, they are endowed with an antiautomorphism, the \emph{standard conjugation}, defined by:
\begin{equation}\label{eq:conjugation}
\bar x=n(x,1)1-x,
\end{equation}
which has the following properties:
\begin{equation}\label{eq:propconjugation}
\bar{\bar x}=x,\quad x+\bar x=n(x,1)1,\quad x\cdot \bar x=\bar x\cdot x=n(x)1,
\end{equation}
for any $x$.

\smallskip

The Hurwitz algebras can always be obtained by the so called \emph{Cayley-Dickson doubling process}. Let $(B,\cdot,n)$ be an associative Hurwitz algebra, and let $\lambda$ be a nonzero scalar in the ground field $k$. Then the direct sum of two copies of $B$: $C=B\oplus Bu$, is endowed with a multiplication and nondegenerate norm that extend those on $B$, and are given by:
\begin{equation}\label{eq:CDdoubling}
\begin{split}
&(a+bu)\cdot (c+du)=(a\cdot c+\lambda \bar d\cdot b)+(d\cdot a+b\cdot\bar c)u,\\
&n(a+bu)=n(a)-\lambda n(b),
\end{split}
\end{equation}
for any $a,b,c,d\in B$. It turns out that $(C,\cdot,n)$ is again a Hurwitz algebra, which is denoted by $C=CD(B,\lambda)$ (see \cite[\S 33]{KMRT}).
Note that the two copies of $B$: $B1$ and $Bu$, in this construction are orthogonal relative to the norm, and that $n(u)=-\lambda$. Whenever the Hurwitz algebra $B$ is itself obtained by the Cayley-Dickson doubling process: $B=CD(A,\mu)$, we will write $C=CD\bigl(CD(A,\mu),\lambda\bigr)=CD(A,\mu,\lambda)$.

\begin{theorem} \textbf{(Generalized Hurwitz Theorem)} (see \cite[page 32]{ZSSS})\newline
Every Hurwitz algebra over a field $k$ is isomorphic to one of the following types:
\begin{romanenumerate}
\item The ground field $k$ if its characteristic is $\ne 2$.
\item A quadratic commutative and associative separable algebra $K(\mu)=k1+kv$, with $v^2=v+\mu$ and $4\mu +1\ne 0$. The norm is given by its generic norm.
\item A quaternion algebra $Q(\mu,\beta)=CD(K(\mu),\beta)$. (These four dimensional algebras are associative but not commutative.)
\item A Cayley algebra $C(\mu,\beta,\gamma)=CD(K(\mu),\beta,\gamma)$. (These eight dimensional algebras are alternative, but not associative.)
\end{romanenumerate}
\end{theorem}

\smallskip

For each possible dimension $\geq 2$, there is up to isomorphism a unique Hurwitz algebra with isotropic norm (that is, there is a nonzero element with zero norm). These are the cartesian product $k\times k$ (with norm $n\bigl((\alpha,\beta)\bigr)=\alpha\beta$), the algebra of $2\times 2$ matrices $\Mat_2(k)$, with norm given by the determinant, and the Cayley algebra $C(k)=CD(\Mat_2(k),-1)$. These, together with the ground field $k$ (in characteristic $\ne 2$), are called the \emph{split} Hurwitz algebras.

The split Cayley algebra $C(k)$ contains nonzero idempotents $e_1,e_2=1-e_1$, so that $e_1\cdot e_2=0=e_2\cdot e_1$, and corresponding Peirce decomposition:
\begin{equation}\label{eq:Peirce}
C(k)=ke_1\oplus ke_2\oplus U\oplus V,
\end{equation}
where $U=\{ x\in C: e_1\cdot x=x=x\cdot e_2\}$ and $V=\{ x\in C: e_2\cdot x=x=x\cdot e_1\}$. Moreover, $n(U)=n(V)=0$, $n(e_1,e_2)=1$, $n(e_1)=n(e_2)=0$, $n(ke_1+ke_2,U+V)=0$, and there are dual bases $\{u_1,u_2,u_3\}$ of $U$ and $\{v_1,v_2,v_3\}$ of $V$ relative to $n$, such that the multiplication table of $C(k)$ in this basis is given by Table \ref{ta:splitCayley}.

\begin{table}[h!]
$$ \vbox{\offinterlineskip
\halign{\hfil$#$\enspace\hfil&#\vreglon
 &\hfil\enspace$#$\enspace\hfil
 &\hfil\enspace$#$\enspace\hfil&#\vregleta
 &\hfil\enspace$#$\enspace\hfil
 &\hfil\enspace$#$\enspace\hfil
 &\hfil\enspace$#$\enspace\hfil&#\vregleta
 &\hfil\enspace$#$\enspace\hfil
 &\hfil\enspace$#$\enspace\hfil
 &\hfil\enspace$#$\enspace\hfil&#\vreglon\cr
 &\omit\hfil\vrule width 1pt depth 4pt height 10pt
   &e_1&e_2&\omit&u_1&u_2&u_3&\omit&v_1&v_2&v_3&\omit\cr
 \noalign{\hreglon}
 e_1&&e_1&0&&u_1&u_2&u_3&&0&0&0&\cr
 e_2&&0&e_2&&0&0&0&&v_1&v_2&v_3&\cr
 &\multispan{11}{\hregletafill}\cr
 u_1&&0&u_1&&0&v_3&-v_2&&-e_1&0&0&\cr
 u_2&&0&u_2&&-v_3&0&v_1&&0&-e_1&0&\cr
 u_3&&0&u_3&&v_2&-v_1&0&&0&0&-e_1&\cr
 &\multispan{11}{\hregletafill}\cr
 v_1&&v_1&0&&-e_2&0&0&&0&u_3&-u_2&\cr
 v_2&&v_2&0&&0&-e_2&0&&-u_3&0&u_1&\cr
 v_3&&v_3&0&&0&0&-e_2&&u_2&-u_1&0&\cr
 &\multispan{12}{\hreglonfill}\cr}}
$$
\caption{The split Cayley algebra $C(k)$}
\label{ta:splitCayley}
\end{table}

As in \cite{EP96}, the basis $\calB=\{e_1,e_2,u_1,u_2,u_3,v_1,v_2,v_3\}$ is said to be a \emph{canonical basis} of $C(k)$.

\begin{remark}\label{re:trilinearmap}
Given the split Cayley algebra $C(k)$ and its decomposition \eqref{eq:Peirce}, the trilinear map $U\times U\times U\rightarrow k$, $(x,y,z)\mapsto n(x,yz)$ is alternating and nonzero. Then, given any basis $\{\tilde u_1,\tilde u_2,\tilde u_3\}$ of $U$ with $n(\tilde u_1,\tilde u_2\cdot \tilde u_3)=1$, its dual basis relative to $n$ in $V$ is $\{\tilde v_1=\tilde u_2\cdot\tilde u_3,\tilde v_2=\tilde u_3\cdot\tilde u_1,\tilde v_3=\tilde u_1\cdot\tilde u_2\}$, and $\{e_1,e_2,\tilde u_1,\tilde u_2,\tilde u_3,\tilde v_1,\tilde v_2,\tilde v_3\}$ is another canonical basis of $C(k)$, that is, it has the same multiplication table (Table \ref{ta:splitCayley}).\hfill\qed
\end{remark}

\medskip

\begin{remark}\label{re:nongroupgrading}
The canonical basis $\calB$ of $C(k)$ gives a grading: $C(k)=\oplus_{b\in\calB}C_b$, with $C_b=kb$ for any $b\in\calB$. And this is not a group grading as $C_{e_1}^2=C_{e_1}$ and $C_{e_2}^2=C_{e_2}$, so the hypothetical grading group would contain two different neutral elements.
\end{remark}

\medskip

It is proved in \cite[Lemma 5]{GradingsOctonions} that any grading group of a Cayley algebra is always abelian. All the possible group gradings, up to equivalence, on Cayley algebras have been classified in \cite{GradingsOctonions}. The classification is summarized in the next result, where the universal grading group of each possible grading is used:

\begin{theorem}\label{th:gradingsCayley}
Let $C=\oplus_{g\in G}C_g$ be a nontrivial group grading of a Cayley algebra over a field $k$, where $G$ is the universal grading group. Then either:

\begin{enumerate}

\item $G=\bZ_2$:\newline
 $C\subo$ is a quaternion subalgebra $Q$ of $C$ and $C\subuno$ is its orthogonal complement relative to the norm. (Hence $C=CD(Q,\alpha)$ for some $0\ne\alpha\in k$ and the $\bZ_2=\bZ/2\bZ$-grading is given by the Cayley-Dickson doubling process.)

\smallskip
\item $G=\bZ_2^2$:\newline
 There is a two dimensional composition subalgebra $K$ of $C$ and elements $x,y\in C$ with $n(x)\ne 0\ne n(y)$, $n(K,x)=0$ and $n(K\oplus Kx,y)=0$ such that
    \[
    C_{(\bar 0,\bar 0)}=K,\quad C_{(\bar 1,\bar 0)}=Kx,\quad C_{(\bar 0,\bar 1)}=Ky,\quad C_{(\bar 1,\bar 1)}=K(xy).
    \]
    (Here $C=CD(K,\beta,\gamma)$ with $\beta=-n(x)$ and $\gamma=-n(y)$ and  the grading is given by the iterated Cayley-Dickson doubling process.)

\smallskip
\item $G=\bZ_2^3$ (the characteristic of $k$ being $\ne 2$):\newline
 There are nonisotropic elements $x,y,z\in C$ such that $n(1,x)=0$, $n(k1+kx,y)=0$, and $n(k1+kx+ky+k(xy),z)=0$ such that the grading is determined by the conditions:
    \[
    C_{(\bar 1,\bar 0,\bar 0)}=kx,\quad C_{(\bar 0,\bar 1,\bar 0)}=ky,\quad C_{(\bar 0,\bar 0,\bar 1)}=kz.
    \]
    (Here $C=CD(k,\alpha,\beta,\gamma)$ with $\alpha=-n(x)$, $\beta=-n(y)$ and $\gamma=-n(z)$ and again the grading is given by the iterated Cayley-Dickson doubling process.)

\smallskip
\item $G=\bZ_3$:\newline
 $C$ is the split Cayley algebra $C(k)$ and
    \[
    \null\qquad C\subo=\espan{e_1,e_2},\quad C\subuno=\espan{u_1,u_2,u_3},\quad C_{\bar 2}=\espan{v_1,v_2,v_3},
    \]
    for a canonical basis of $C(k)$.

\smallskip
\item $G=\bZ_4$:\newline
  $C$ is the split Cayley algebra $C(k)$ with a canonical basis such that:
    \[
    \begin{aligned}
    C\subo&=\espan{e_1,e_2},&\quad C\subuno&=\espan{u_1,u_2},\\
    C_{\bar 2}&=\espan{u_3,v_3},& C_{\bar 3}&=\espan{v_1,v_2}.
    \end{aligned}
    \]

\smallskip
\item $G=\bZ$ ($3$-grading):\newline
 $C$ is the split Cayley algebra $C(k)$ with a canonical basis such that:
    \[
    \null\qquad C_0=\espan{e_1,e_2,u_3,v_3},\quad C_1=\espan{u_1,v_2},\quad C_{-1}=\espan{u_2,v_1}.
    \]

\smallskip
\item $G=\bZ$, ($5$-grading):\newline
 $C$ is the split Cayley algebra $C(k)$ with a canonical basis such that:
    \[
    \begin{aligned}
    C_0&=\espan{e_1,e_2},&\quad C_1&=\espan{u_1,u_2},&\quad C_2&=\espan{v_3},\\
    && C_{-1}&=\espan{v_1,v_2},& C_{-2}&=\espan{u_3}.
    \end{aligned}
    \]

\smallskip
\item $G=\bZ^2$:\newline
 $C$ is the split Cayley algebra $C(k)$ with a canonical basis such that:
    \[
    \begin{aligned}
    &&C_{(0,0)}&=\espan{e_1,e_2},&&\\
    C_{(1,0)}&=\espan{u_1},&\ C_{(0,1)}&=\espan{u_2},&\quad C_{(1,1)}&=\espan{v_3},\\
     C_{(-1,0)}&=\espan{v_1},&\ C_{(0,-1)}&=\espan{v_2},&\ C_{(-1,-1)}&=\espan{u_3}.
     \end{aligned}
    \]

\smallskip
\item $G=\bZ\times\bZ_2$:\newline
 $C$ is the split Cayley algebra $C(k)$ with a canonical basis such that:
    \[
    \begin{aligned}
    C_{(0,\bar 0)}&=\espan{e_1,e_2},&\quad C_{(0,\bar 1)}&=\espan{u_3,v_3},\\
     C_{(1,\bar 0)}&=\espan{u_1},&\ C_{(1,\bar 1)}&=\espan{v_2},\\
      C_{(-1,\bar 0)}&=\espan{v_1},&\ C_{(-1,\bar 1)}&=\espan{u_2}.
     \end{aligned}
    \]

\end{enumerate}
\end{theorem}

\medskip

\begin{remark}\label{re:coarseningCayley}
If the characteristic of the ground field $k$ is $\ne 2$, all the gradings of a Cayley algebra $C$ are coarsenings of either a $\bZ_2^3$-grading or a $\bZ^2$-grading.\hfill\qed
\end{remark}

\bigskip

The arguments used in \cite{GradingsOctonions} give too the possible gradings on Hurwitz algebras of dimension $4$. Alternatively, given a graded quaternion algebra $Q=\oplus_{g\in G}Q_g$, then the Cayley algebra $C=CD(Q,1)=Q\oplus Qu$  ($u^2=1$) is $G\times \bZ_2$-graded with $C_{(g,\bar 0)}=Q_g$ and $C_{(g,\bar 1)}=Q_gu$. This allows to compute easily all the possibilities:

\begin{corollary}\label{co:gradingsquaternion}
Let $Q=\oplus_{g\in G}Q_g$ be a nontrivial group grading of a quaternion algebra over a field $k$, where $G$ is the universal grading group. Then either:
\begin{enumerate}
\item $G=\bZ_2$:\newline
 $Q\subo$ is a composition two dimensional subalgebra $K$ of $Q$ and $Q\subuno$ is its orthogonal complement relative to the norm. (Hence $Q=CD(K,\alpha)$ for some $0\ne\alpha\in k$ and the $\bZ_2$-grading is given by the Cayley-Dickson doubling process.)

\item $G=\bZ_2^2$ (the characteristic of $k$ being $\ne 2$):\newline
 There are nonisotropic elements $x,y\in Q$ such that $n(1,x)=0$ and $n(k1+kx,y)=0$ such that the grading is determined by the conditions:
    \[
    Q_{(\bar 1,\bar 0)}=kx,\quad Q_{(\bar 0,\bar 1)}=ky.
    \]
    (Here $Q=CD(k,\alpha,\beta)$ with $\alpha=-n(x)$, $\beta=-n(y)$, and again the grading is given by the iterated Cayley-Dickson doubling process.)

\item $G=\bZ$ ($3$-grading):\newline
 $Q$ is, up to isomorphism, the split quaternion algebra $\Mat_2(k)$ and:
    \[
    \null\qquad Q_0=\espan{\left(\begin{smallmatrix} 1&0\\ 0&0\end{smallmatrix}\right), \left(\begin{smallmatrix} 0&0\\ 0&1\end{smallmatrix}\right)},\quad
    Q_1=\espan{\left(\begin{smallmatrix} 0&1\\ 0&0\end{smallmatrix}\right)},\quad
    Q_{-1}=\espan{\left(\begin{smallmatrix} 0&0\\ 1&0\end{smallmatrix}\right)}.
    \]

\end{enumerate}
\end{corollary}

\medskip

As for the two dimensional Hurwitz algebras, any such algebra $K$ can only be nontrivially graded if the characteristic of the ground field $k$ is not $2$, and then $K=CD(k,\alpha)=k1\oplus ku$, with $u^2=\alpha$, and this provides a $\bZ_2$-grading. This is the only nontrivial possibility.

\bigskip

\section{Symmetric composition algebras}\label{se:SymCompo}

A new class of composition algebras has been considered lately by a number of authors  (\cite{Pet69, Okubo78,Okubo95,OkuboMyung80,OkuboOsborn1,OkuboOsborn2,EM91,EM93}).

A composition algebra $(S,*,n)$ is said to be \emph{symmetric} if the polar form of its norm is associative:
\begin{equation}\label{eq:normassociative}
n(x*y,z)=n(x,y*z),
\end{equation}
for any $x,y,z\in S$. This condition is equivalent (see \cite[Lemma 2.3]{OkuboOsborn2} or \cite[(34.1)]{KMRT}) to the condition:
\begin{equation}\label{eq:xyx}
(x*y)*x=n(x)y=x*(y*x),
\end{equation}
for any $x,y\in S$.

\smallskip

The first examples of symmetric composition algebras are given by the so called \emph{para-Hurwitz algebras} \cite{OkuboMyung80}. Given a Hurwitz algebra $(C,\cdot,n)$, its para-Hurwitz counterpart is the composition algebra $(C,\bullet,n)$, with
\[
x\bullet y=\bar x\cdot \bar y,
\]
for any $x,y\in C$, where $x\mapsto \bar x$ is the standard conjugation in the Hurwitz algebra $C$. This algebra will be denoted by $\bar C$ for short. Note that the unity of $(C,\cdot,n)$ becomes a \emph{para-unit} in $\bar C$, that is, an element $e$ such that $e\bullet x=x\bullet e=n(e,x)e-x$ for any $x$. If the dimension is at least $4$, the para-unit is unique, and it is the unique idempotent that spans the commutative center of the para-Hurwitz algebra.

\smallskip

A slight modification of the above procedure was considered previously by Petersson (\cite{Pet69}) as follows:

Let $\tau$ be an automorphism of a Hurwitz algebra $(C,\cdot,n)$ with $\tau^3=1$, and consider the new multiplication defined on $C$ by means of:
\begin{equation}\label{eq:Petersson}
x*y=\tau(\bar x)\cdot\tau^2(\bar y),
\end{equation}
for any $x,y\in C$. Then the algebra $(C,*,n)$ is a symmetric composition algebra (a \emph{Petersson algebra}), which will be denoted by $\bar C_\tau$ for short.

Consider a canonical basis $\{e_1,e_2,u_1,u_2,u_3,v_1,v_2,v_3\}$ of the split Cayley algebra $C(k)$ as in Table \ref{ta:splitCayley}. Then the linear map $\tau_{st}:C(k)\rightarrow C(k)$ determined by the conditions:
\begin{equation}\label{eq:taust}
\tau_{st}(e_i)=e_i,\ i=1,2;\quad\tau_{st}(u_i)=u_{i+1},\ \tau_{st}(v_i)=v_{i+1}\ \textrm{(indices modulo $3$)},
\end{equation}
is clearly an order $3$ automorphism of $C(k)$. (Here ``$st$'' stands for \emph{standard}.) The associated Petersson algebra $P_8(k)=\overline{C(k)}_{\tau_{st}}$ is called the \emph{pseudo-octonion algebra} over the field $k$ (see \cite[p.~1095]{EP96}). This definition extends and unifies previous definitions by Okubo \cite{Okubo78} and Okubo and Osborn \cite{OkuboOsborn2}.

The forms of $P_8(k)$ are called \emph{Okubo algebras} (see \cite{EM90}).

\smallskip

The classification of the symmetric composition algebras was obtained in \cite{EM93} (see also \cite[Theorem 1]{EldTwisted} and \cite[(34.37)]{KMRT}) over fields of characteristic $\ne 3$, and in \cite{Eld97} in characteristic $3$. It turns out that, apart from some forms of the two dimensional para-Hurwitz algebra, any symmetric composition algebra is either a para-Hurwitz or an Okubo algebra.

\smallskip

Even though the classification follows different paths according to the characteristic being $3$ or different from $3$, the following unifying result was obtained in \cite[Theorem 7]{EldTwisted}:

\begin{theorem}\label{th:unified}
For any Okubo algebra $(S,*,n)$ with isotropic norm over a field $k$, there are nonzero scalars $\alpha,\beta\in k$ and a basis $\{x_{ij}: -1\leq i,j\leq 1,\ (i,j)\ne (0,0)\}$ such that the multiplication table is given by Table \ref{ta:Oalphabeta}.
\end{theorem}

\begin{table}[h!]
{\tiny
$$
\vcenter{\vbox{\offinterlineskip \halign{\hfil#\hfil\kern 1pt
&#\vreglon& \kern 1pt\hfil#\hfil&
\hfil#\hfil\kern 1pt&
#\vregleta&
\kern 1pt\hfil#\hfil&
\hfil#\hfil\kern 1pt&
#\vregleta&
\kern 1pt\hfil#\hfil\kern -2pt&
\kern -2pt\hfil#\hfil\kern 1pt&
#\vregleta&
\kern 1pt\hfil#\hfil\kern -2pt&
\kern -2pt\hfil#\hfil\kern 1pt&#\vreglon\cr
&&
    $x_{1,0}$&$x_{-1,0}$&\omit&$x_{0,1}$&$x_{0,-1}$&\omit
    &$x_{1,1}$&$x_{-1,-1}$&\omit&$x_{-1,1}$&$x_{1,-1}$&\omit\cr
\noalign{\hreglon}\cr
\bigstrut
   $x_{1,0}$&&$-\alpha x_{-1,0}$&$0$&
   &$0$&$x_{1,-1}$&
   &$0$&$x_{0,-1}$&
			&$0$&$\alpha x_{-1,-1}$&\cr
\bigstrut
   $x_{-1,0}$&&$0$&$-\alpha^{-1}x_{1,0}$&
   &$x_{-1,1}$&$0$&
   &$x_{0,1}$&$0$&
			&$\alpha^{-1}x_{1,1}$&$0$&\cr
&\multispan{13}\hregletafill\cr
\bigstrut
   $x_{0,1}$&&$x_{1,1}$&$0$&
   &$-\beta x_{0,-1}$&$0$&
   &$\beta x_{1,-1}$&$0$&
			&$0$&$x_{1,0}$&\cr
\bigstrut
   $x_{0,-1}$&&$0$&$x_{-1,-1}$&
   &$0$&$-\beta^{-1} x_{0,1}$&
   &$0$&$\beta^{-1}x_{-1,1}$&
			&$x_{-1,0}$&$0$&\cr
&\multispan{13}\hregletafill\cr
\bigstrut
   $x_{1,1}$&&$\alpha x_{-1,1}$&$0$&
   &$0$&$x_{1,0}$&
   &$-(\alpha\beta)x_{-1,-1}$&$0$&
			&$\beta x_{0,-1}$&$0$&\cr
\bigstrut
   $x_{-1,-1}$&&$0$&$\alpha^{-1}x_{1,-1}$&
   &$x_{-1,0}$&$0$&
   &$0$&$-(\alpha\beta)^{-1}x_{1,1}$&
			&$0$&$\beta^{-1}x_{0,1}$&\cr
&\multispan{13}\hregletafill\cr
\bigstrut
   $x_{-1,1}$&&$x_{0,1}$&$0$&
   &$\beta x_{-1,-1}$&$0$&
   &$0$&$\alpha^{-1}x_{1,0}$&
			&$-\alpha^{-1}\beta x_{1,-1}$&$0$&\cr
\bigstrut
   $x_{1,-1}$&&$0$&$x_{0,-1}$&
   &$0$&$\beta^{-1}x_{1,1}$&
   &$\alpha x_{-1,0}$&$0$&
			&$0$&$-\alpha\beta^{-1}x_{-1,1}$&\cr
&\multispan{13}\hreglonfill\cr}}}
$$}
\caption{$\calO_{\alpha,\beta}$}
\label{ta:Oalphabeta}
\end{table}

\smallskip

The Okubo algebra with the multiplication table given in Table \ref{ta:Oalphabeta} will be denoted by $\calO_{\alpha,\beta}$. It must be remarked here that over fields of characteristic $3$, the norm of any Okubo algebra is isotropic (see \cite[Lemma 3.7]{EP96} and \cite[Corollary 3.4]{Eld97}). The same happens over fields of characteristic $\ne 3$ containing the cubic roots of $1$ (\cite[Corollary 3.6]{EP96}).

\begin{remark}\label{re:Z3naturalgrading}
The Okubo algebra $\calO_{\alpha,\beta}$ is naturally $\bZ_3^2$ graded, with
\begin{equation}\label{eq:Z32standard}
(\calO_{\alpha,\beta})_{(\bar 1,\bar 0)}=k x_{1,0},\ \text{and}\  (\calO_{\alpha,\beta})_{(\bar 0,\bar 1)}=k x_{0,1}.
\end{equation}

This grading will be referred to as the \emph{standard $\bZ_3^2$-grading} of $\calO_{\alpha,\beta}$, and will play an important role later on.

\smallskip

Also, by coarsening this grading, there appears the $\bZ_3$-grading where for $i=0,1,2$,
\begin{equation}\label{eq:Z3standard}
(\calO_{\alpha,\beta})_{\bar \imath}=\oplus_{j=0}^2(\calO_{\alpha,\beta})_{(\bar \jmath,\bar \imath)}.
\end{equation}
This will be called the \emph{standard $\bZ_3$-grading} of $\calO_{\alpha,\beta}$. \hfill\qed
\end{remark}

\medskip

In the determination of the gradings of the symmetric composition algebras, it will be important to be able to recognize the Okubo algebras $\calO_{\alpha,\beta}$. The following results, which have their own independent interest, are aimed at this objective.

Given an element $x$ of an algebra, $\alg{x}$ will denote the subalgebra generated by $x$.

\begin{proposition}\label{pr:interesting}
Let $(S,*,n)$ be an Okubo algebra over a field $k$ of characteristic $\ne 3$ containing a nonzero element $x\in S$ such that $n(x)=0\ne n(x,x*x)$. Then there is a nonzero element $y\in S$, with $n(y)=0\ne n(y,y*y)$, $x*y=0$ and $n(\alg{x},\alg{y})=0$.
\end{proposition}
\begin{proof}
Since $(x*x)*x=n(x)x=0=x*(x*x)$ because of \eqref{eq:xyx}, and
\[
(x*x)*(x*x)=-((x*x)*x)*x+n(x,x*x)x=n(x,x*x)x,
\]
the subalgebra generated by the element $x$ is $\alg{x}=kx+k(x*x)$, which is a composition subalgebra of $S$ (that is, the restriction of the norm is nondegenerate). Let $\alpha=n(x,x*x)$.

Assume first that $\alpha\not\in k^3$, and consider the element $p=x+\alpha^{-1}x*x$, whose norm is $n(p)=1$. Let $l_p$ and $r_p$ denote the left and right multiplications by $p$, which are isometries of $(S,n)$, and consider the Cayley algebra $(S,\cdot,n)$ with multiplication given by the equation:
\[
a\cdot b=l_p(a)* r_p(b),
\]
for any $a,b\in S$, whose unity is the element
\[
q=p*p=x*x+\alpha^{-2}(x*x)*(x*x)=\alpha^{-1}x+x*x.
\]
(Note that for any $b$, $q\cdot b=(p*q)*(b*p)=(p*(p*p))*(b*p)=n(p)p*(b*p)=n(p)^2b=b=b\cdot q$.)

It follows that
\[
\begin{split}
x\cdot x&=(p*x)*(x*p)=(x*x)*(x*x)=\alpha x,\\
(x*x)\cdot(x*x)&=(p*(x*x))*((x*x)*p)=x*x,\\
x\cdot(x*x)&=(p*x)*((x*x)*p)=(x*x)*x=0=(x*x)\cdot x,
\end{split}
\]
so that the elements $e_1=\alpha^{-1}x$ and $e_2=x*x$ are idempotents of the Cayley algebra  $(S,\cdot ,n)$ whose sum is the unity $q$. Consider the associated Peirce decomposition:
\[
S=ke_1\oplus ke_2\oplus U\oplus V,
\]
where $U=\{z\in S: e_1\cdot z=z=z\cdot e_2\}$ and $V=\{z\in S: e_2\cdot z=z=z\cdot e_1\}$.

Note that for any $z\in \alg{x}^\perp=U\oplus V$, we get:
\begin{equation}\label{eq:ze10}
\begin{split}
z\cdot e_1=0&\Leftrightarrow (p*z)*(e_1*p)=0\\
 &\Leftrightarrow ((x+\alpha^{-1}x*x)*z)*(\alpha^{-1}x*x)=0\\
 &\Leftrightarrow (x*z)*(\alpha^{-1}x*x)=0\quad\text{(as $n(x*x)=0$)}\\
 &\Leftrightarrow (x*z)*p=0\quad\text{(as $(x*z)*x=0$ since $n(x)=0$)}\\
 &\Leftrightarrow x*z=0\quad\text{(since $n(p)=1$, so $r_p$ is an isometry).}
\end{split}
\end{equation}
Since $U=\{z\in \alg{x}^\perp: z\cdot e_1=0\}$, we obtain
\[
U=\{z\in\alg{x}^\perp: x*z=0\}.
\]
Also, for any $z\in U$, using \eqref{eq:xyx}, we get
\[
(z*p)\cdot e_1=z*(e_1*p)=\alpha^{-1}z*(x*x)=-\alpha^{-1}x*(x*z)=0,
\]
so $r_p(U)\subseteq U$. Now, for any $0\ne u\in U$:
\[
r_p^3(u)=((u*p)*p)*p=-(p*p)*(u*p)=-p\cdot u=-(\alpha e_1+\alpha^{-1}e_2)\cdot u=-\alpha u,
\]
so, since we are assuming $-\alpha\not\in k^3$ and hence the polynomial $X^3+\alpha$ is irreducible, the minimal polynomial of $u$ relative to the endomorphism $r_p\vert_{U}$ is $X^3+\alpha$, and $\{u,u*p,(u*p)*p\}$ is a basis of $U$. But the alternating trilinear map $U\times U\times U\rightarrow k$: $(a,b,c)=n(a,b\cdot c)$, is nonzero (see Remark \ref{re:trilinearmap}), so we obtain:
\[
\begin{split}
0&\ne n(u,(u*p)\cdot ((u*p)*p))\\
&=n\bigl(u,(p*(u*p)*(((u*p)*p)*p)\bigr)\\
&=n(u,u*(-\alpha u))=-\alpha n(u,u*u),
\end{split}
\]
and $u*u\in U\cdot U\subseteq V$. Besides, $x*u=0$, because of \eqref{eq:ze10}, as $u\in U$. Hence it is enough to take the element $y=u$, as $\alg{y}=\espan{u,u*u}\subseteq U+V\subseteq \alg{x}^\perp$ (the orthogonal to $\alg{x}$ relative to the norm $n$).

\smallskip

Finally, assume that $\alpha\in k^3$ and take $\beta\in k$ with $\alpha=\beta^3$. Then, changing $x$ to $\beta^{-1}x$, we may assume that $\alpha=1$, that is $n(x,x*x)=1$. Then $e=x+x*x$ is a nonzero idempotent of $(S,*,n)$ and the linear map $\tau:S\rightarrow S$: $a\mapsto n(a,e)e-e*a$ is an automorphism of both $(S,*,n)$ and of $(S,\cdot ,n)$, where
\[
a\cdot b=(e*a)*(b*e),
\]
for any $a,b\in S$ (see \cite[Theorem 2.5]{EP96}). Moreover, $\tau^3=1$ holds.

Note that $(S,\cdot,n)$ is a Cayley algebra with unity $e$ and that the multiplication $*$ becomes
\[
a*b=\tau(\bar a)\cdot\tau^2(\bar b),
\]
for any $a,b\in S$, so that $(S,*,n)$ is a Petersson algebra.

In this case, $e_1=x$ and $e_2=x*x=e-e_1$ are idempotents of $(S,\cdot,n)$ with $e_1\cdot e_2=e_2\cdot e_1=0$ and we may consider again the associated Peirce decomposition $S=ke_1\oplus ke_2\oplus U\oplus V$ as before. Then $\tau(e_1)=e_1$, $\tau(e_2)=e_2$, so both $U$ and $V$ are invariants under $\tau$. Since the characteristic of $k$ is not $3$, the result in \cite[Theorem 3.5]{EP96} forces the minimal polynomial of the restrictions $\tau\vert_U$ and $\tau\vert_V$ to be exactly $X^3-1$. Otherwise, the algebra would be para-Hurwtiz. Now, take an element $u\in U$ such that the minimal polynomial of $u$ relative to $\tau$ is $X^3-1$. Then $\{u,\tau(u),\tau^2(u)\}$ is a basis of $U$, so
\[
0\ne n(u,\tau(u)\cdot\tau^2(u)=n(u,u*u).
\]
As before, it is enough to take the element $y=u$.
\end{proof}

\begin{remark}\label{re:interesting}
The Proposition above and its proof are valid in characteristic $3$, provided that either $\alpha\not\in k^3$, or $\alpha\in k^3$ and the minimal polynomial of $\tau\vert_U$ is $X^3-1$ (notation as in the proof above). \hfill\qed
\end{remark}

\smallskip

The next result is inspired by (and extends) \cite[Proposition 4.1]{Eld97}.

\begin{theorem}\label{th:bonito}
Let $(S,*,n)$ be an Okubo algebra over an arbitrary field $k$, and let $x,y\in S$ be two elements satisfying the conditions:
\[
n(x)=n(y)=0,\quad n(x,x*x)\ne 0\ne n(y,y*y),\quad n(\alg{x},\alg{y})=0.
\]
Then either $x*y=0$ or $y*x=0$ but not both.

Moreover, assuming $x*y=0$, the set $\{x,x*x,y,y*y,y*x,(y*y)*(x*x),x*(y*y),(x*x)*y\}$ is a basis of $S$, and the multiplication table in this basis is completely determined and only depends on $\alpha=n(x,x*x)$ and $\beta=n(y,y*y)$.
\end{theorem}
\begin{proof}
Since $n(x)=0\ne n(x,x*x)$, and because of \eqref{eq:xyx}, it follows that  $\alg{x}$ is spanned by $x$ and $x*x$, and similarly $\alg{y}$ is spanned by $y$ and $y*y$. The associativity of the norm gives:
\[
n\left(\alg{x}*\alg{y},\alg{x}\right)\subseteq n\left(\alg{y},\alg{x}*\alg{x}\right)=0,
\]
and, in the same vein:
\begin{multline}\label{eq:bonito1}
n\left(\alg{x}*\alg{y},\alg{x}+\alg{y}\right)=0\\ =
 n\left(\alg{y}*\alg{x},\alg{x}+\alg{y}\right).
\end{multline}
Besides, using the linearization of \eqref{eq:xyx} we obtain:
\begin{equation}\label{eq:bonito2}
\begin{split}
n\bigl(\alg{x}*\alg{y}&,\alg{y}*\alg{x}\bigr)\\
  &\subseteq
  n\left(\alg{x},\alg{y}*\bigl(\alg{y}*\alg{x}\bigr)\right)\\
  &\subseteq n\left(\alg{x},\alg{x}*\bigl(\alg{y}*\alg{y}\right)\\
  &\subseteq n\left(\alg{x}*\alg{x},\alg{y}*\alg{y}\right)\\
  &=n\left(\alg{x},\alg{y}\right)=0.
\end{split}
\end{equation}
As $n(x+x*x)=n(x,x*x)\ne 0$, the left and right multiplications by $x+x*x$ are similarities, so $(x+x*x)*\alg{y}$ and $\alg{y}*(x+x*x)$ are orthogonal (because of \eqref{eq:bonito2}) nondegenerate two dimensional subspaces of $\left(\alg{x}+\alg{y}\right)^\perp$. By dimension count, it turns out that
\begin{equation}\label{eq:bonito3}
\{x,x*x,y,y*y,(x+x*x)*y,(x+x*x)*(y*y),y*(x+x*x),(y*y)*(x+x*x)\}
\end{equation}
is a basis of $S$.

Now, the element $x*y\in \alg{x}*\alg{y}$ is orthogonal to $x$, $x*x$, $y$, $y*y$, $y*(x+x*x)$ and $(y*y)*(x+x*x)$ because of \eqref{eq:bonito1} and \eqref{eq:bonito2}, and
\[
n(x*y,(x+x*x)*y)=n(x,y*((x+x*x)*y))=n(y)n(y,x+x*x)=0.
\]
Thus $x*y$ belongs to the orthogonal subspace to the seven dimensional space spanned by the basic elements $x,x*x,y,y*y,(x+x*x)*y, y*(x+x*x),(y*y)*(x+x*x)$, which is one dimensional and spanned by $(x+x*x)*y$. Hence there is a scalar $\alpha\in k$ such that $x*y=\alpha(x+x*x)*y$, or $\bigl((\alpha -1)x+\alpha x*x)*y=0$. This implies that the element $(\alpha-1)x+x*x$ is isotropic, so
\[
0=n\bigl((\alpha-1)x+\alpha x*x\bigr)=\alpha(\alpha-1)n(x,x*x).
\]
Therefore, either $\alpha=0$ and hence $x*y=0$, or $\alpha=1$ and $(x*x)*y=0$. In the latter case,
\begin{multline*}
(y*x)*(x+x*x)=(y*x)*x+(y*x)*(x*x)\\
  =-(x*x)*y-((x*x)*x)*y=-n(x)x*y=0,
\end{multline*}
where we have used \eqref{eq:xyx} and the fact that $n(x)=0$. Hence $y*x=0$.

Therefore, either $x*y=0$ or $y*x=0$. Permuting $x$ and $y$ if necessary, it can be assumed that $x*y=0$. But this forces $y*(x*x)=-x*(x*y)=0$, so $0\ne y*(x+x*x)=y*x$ and $0\ne (x+x*x)*y=(x*x)*y$. Also, $(y*y)*x=-(x*y)*y=0$, so $0\ne (y*y)*(x+x*x)=(y*y)*(x*x)$, while $(x*x)*(y*y)=-y*(y*(x*x))=0$, so $0\ne (x+x*x)*(y*y)=x*(y*y)$. Hence the basis in \eqref{eq:bonito3} becomes, after reordering it, the basis
\begin{equation}\label{eq:basisxy}
\{x,x*x,y,y*y,y*x,(y*y)*(x*x),x*(y*y),(x*x)*y\}
\end{equation}
as desired.

It is now easy to check that the multiplication constants in this basis are completely determined (see \cite[proof of Theorem 4.3]{Eld97}) in terms of $\alpha=n(x,x*x)$ and $\beta=n(y,y*y)$. For instance,
\begin{multline*}
(y*x)*((x*x)*y)\\
 =-y*((x*x)*(y*x))=-y*\bigl(n(x,x*x)y-x*(y*(x*x))\bigr)=-\alpha y*y,
\end{multline*}
or
\begin{multline*}
((x*x)*y)*(y*x)\\
  =n(x*x,y*x)y-((y*x)*y)*(x*x)=n(x)n(x,y)y-n(y)x*(x*x)=0.
\end{multline*}
\end{proof}

\begin{corollary}\label{co:bonito}
Let $(S,*,n)$ be an Okubo algebra over an arbitrary field $k$ containing elements $x,y\in S$ satisfying the conditions in Theorem \ref{th:bonito}:
\[
n(x)=n(y)=0,\quad n(x,x*x)\ne 0\ne n(y,y*y),\quad n(\alg{x},\alg{y})=0,
\]
and such that $x*y=0$. Then $S$ is isomorphic to the Okubo algebra $\calO_{\alpha,\beta}$, with $\alpha=n(x,x*x)$ and $\beta=n(y,y*y)$, under an isomorphism that takes $x$ to $-x_{1,0}$ and $y$ to $-x_{0,1}$.

In particular, the Okubo algebra $S$ is graded over $\bZ_3^2$ with $x\in S_{(\bar 1,\bar 0)}$ and $y\in S_{(\bar 0,\bar 1)}$.
\end{corollary}
\begin{proof}
The Okubo algebra $\calO_{\alpha,\beta}$ is generated by the elements $-x_{1,0}$ and $-x_{0,1}$, which satisfy the same properties of $x$ and $y$ in the hypotheses of the Corollary. The result then follows by the uniqueness in Theorem \ref{th:bonito}.
\end{proof}

The two previous results imply easily the following description of the Okubo algebras with nonzero idempotents and isotropic norm:

\begin{theorem}\label{th:isotropic}
The Okubo algebras with isotropic norm and nonzero idempotents are precisely the algebras $\calO_{1,\beta}$ with $0\ne \beta\in k$.
\end{theorem}
\begin{proof}
The element $e=-(x_{1,0}+x_{-1,0})$ in $\calO_{1,\beta}$ is a nonzero idempotent. Conversely, let $(S,*,n)$ be an Okubo algebra with a nonzero idempotent $e$. Because of \eqref{eq:xyx}, its norm is $1$. If the characteristic of $k$ is $3$, the result follows from \cite[Theorem B]{EP96} or \cite[Theorem 5.(3)a and (3)b]{Eld97}. On the other hand, if the characteristic is $\ne 3$, the result in \cite[Lemma 3]{EldTwisted} shows the existence of an element $0\ne x\in S$ such that $n(x)=0$ and $n(x,x*x)=1$. Now, Proposition \ref{pr:interesting} shows the existence of another element $y$ with $n(y)=0\ne n(y,y*y)=\beta$ and $n(\alg{x},\alg{y})=0$. Theorem \ref{th:bonito} shows that either $x*y=0$ or $y*x=0$. In case $x*y=0$, the result follows from Corollary \ref{co:bonito}, while if $y*x=0$, $(x*x)*y=-(y*x)*x=0$, while $n(x*x,(x*x)*(x*x))=n(x*x,n(x,x*x)x-((x*x)*x)*x)=n(x*x,x)=1$, so that we can substitute $x$ by $x*x$ and apply again Corollary \ref{co:bonito}.
\end{proof}

\smallskip

As mentioned before Remark \ref{re:Z3naturalgrading}, over fields of characteristic $3$, the norm of any Okubo algebra is isotropic, so the previous Theorem gives the description of any Okubo algebra with nonzero idempotents over these fields. On the other hand, the Okubo algebras with nonzero idempotents over fields of characteristic $\ne 3$, no matter wether its norm is isotropic or not, are easily described as follows:

\begin{proposition}\label{pr:idempotentsno3}
Let $(S,*,n)$ be an Okubo algebra with a nonzero idempotent over a field $k$ of characteristic $\ne 3$. Then there is a quaternion algebra $Q$ which contains a two dimensional subalgebra $K=k1+kw$, with $w^2+w+1=0$ (that is, $K$ is isomorphic to the commutative separable algebra $k[X]/(X^2+X+1)$), and a nonzero scalar $\alpha\in k$, such that $(S,*,n)$ is isomorpic to the Petersson algebra $\bar C_\tau$, where $C$ is the Cayley algebra $C=CD(Q,\alpha)=Q\oplus Qu$, $u^2=\alpha$, and $\tau$ is the order $3$ automorphism of $C$ determined by the conditions:
\[
\tau(q)=q\ \textrm{for any $q\in Q$},\qquad \tau(u)=wu.
\]
\end{proposition}
\begin{proof}
Because of \cite[Therems 2.5 and 3.5]{EP96}, we know that the Okubo algebra $(S,*,n)$ is isomorphic to a Petersson algebra $\bar C_\tau$, where $C$ is a Cayley algebra and $\tau$ an order $3$ automorphism of $C$ such that $Q=\{x\in C: \tau(x)=x\}$ is a quaternion subalgebra. Take $u\in Q^\perp$ with $n(u)=-\alpha\ne 0$, so that $C=Q\oplus Qu=CD(Q,\alpha)$. Besides, $\tau(Q^\perp)=Q^\perp=Qu$, so that $\tau(u)=wu$ for some $w\in Q$. Since $u^2=\alpha\in k$, it follows that $w\not\in k$, and since the minimal polynomial of the restriction of $\tau$ to $Q^\perp$ is $X^2+X+1$, it follows that $w^2+w+1=0$, and the result follows.
\end{proof}

\begin{remark}\label{re:Z2gradingsOkubo}
Let $C$, $Q$, $K=k1+kw$, $u$, $\tau$ and $S=\bar C_\tau$ as in the previous Proposition. Then the Okubo algebra $S$ is naturally graded in these two ways:
\begin{itemize}
\item $\bZ_2^2$-graded with $S_{(\bar 0,\bar 0)}=K$, $S_{(\bar 1,\bar 0)}=K^\perp\cap Q$, $S_{(\bar 0,\bar 1)}=Ku$, and $S_{(\bar 1,\bar 1)}=(K^\perp\cap Q)u$.

\item $\bZ_2$-graded, with $S\subo=Q$ and $S\subuno=Qu$. \hfill\qed
\end{itemize}
\end{remark}

\smallskip

Before studying the gradings on the Okubo algebras, let us give a couple of presentations of the pseudo-octonion algebra $P_8(k)$ which show some interesting gradings on this algebra. Actually, $P_8(k)$ is defined as the composition algebra $C(k)_{\tau_{st}}$ given en \eqref{eq:taust}. This definition of the pseudo-octonion algebra highlights a natural $\bZ_3$-grading, inherited from the $\bZ_3$-grading of $C(k)$
in Theorem \ref{th:gradingsCayley}, which is given by the Peirce decomposition:
\[
P_8(k)\subo=\espan{e_1,e_2},\quad P_8(k)\subuno=\espan{u_1,u_2,u_3},\quad P_8(k)_{\bar 2}=\espan{v_1,v_2,v_3}.
\]

Given a canonical basis of the split Cayley algebra $C(k)$, consider the new order $3$ automorphism $\tau_{nst}$ defined by:
\begin{equation}\label{eq:taunst}
\begin{aligned}
\tau_{nst}(e_1)&=e_1,& \tau_{nst}(e_2)&=e_2,&&\\
\tau_{nst}(u_1)&=u_2,& \tau_{nst}(u_2)&=-u_1-u_2,& \quad\tau_{nst}(u_3)&=u_3,\\
\tau_{nst}(v_1)&=-v_1+v_2,& \quad\tau_{nst}(v_2)&=-v_1,& \tau_{nst}(v_3)&=v_3.
\end{aligned}
\end{equation}

\begin{proposition}\label{pr:taunst}
The Petersson algebra $\overline{C(k)}_{\tau_{nst}}$ is isomorphic to the pseudo-octonion algebra $P_8(k)$.
\end{proposition}
\begin{proof}
This is straightforward if the characteristic is $\ne 3$ by Corollary \ref{co:bonito}. Just take $x=e_1$ and $y=u_1+\frac{1}{3}u_3$ in the split Cayley algebra $C(k)$. Then in the Petersson algebra $\overline{C(k)}_{\tau_{nst}}$ we have:
\[
\begin{split}
x*x&=\bar e_1 \bar e_1=e_2  e_2=e_2,\\[2pt]
y*y&= \tau_{nst}\Bigl(\overline{u_1+\frac{1}{3}u_3}\Bigr)
         \tau_{nst}^2\Bigl(\overline{u_1+\frac{1}{3}u_3}\Bigr)\\
   &=\Bigl(-u_2-\frac{1}{3}u_3\Bigr) \Bigl(u_1+u_2-\frac{1}{3}u_3\Bigr)\\
   &=\frac{2}{3}v_1-\frac{1}{3}v_2+v_3,\\[2pt]
x*y&=e_2  (u_1+u_2-\frac{1}{3}u_3)=0,
\end{split}
\]
so that $\alg{x}$ and $\alg{y}$ are orthogonal two dimensional composition subalgebras, $n(x)=0=n(y)$, and hence, due to Corollary \ref{co:bonito}, $\overline{C(k)}_{\tau_{nst}}$ is isomorphic to $\calO_{\alpha,\beta}$ with
\[
\begin{split}
\alpha&=n(x,x*x)=n(e_1,e_2)=1,\\
\beta&=n(y,y*y)=n\left(u_1+\frac{1}{3}u_3,\frac{2}{3}v_1-\frac{1}{3}v_2+v_3\right)=1.
\end{split}
\]
Then $\overline{C(k)}_{\tau_{nst}}$ is isomorphic to $\calO_{1,1}$ which, in turn, is isomorphic to the pseudo-octonion algebra.

\smallskip

If the characteristic of our ground field $k$ is $3$, then note that $S=\overline{C(k)}_{\tau_{nst}}$ is $\bZ_2$-graded with $S\subo=\espan{e_1,e_2,u_3,v_3}$ and $S\subuno=\espan{u_1,u_2,v_1,v_2}$. Hence, if $S$ were a para-Hurwitz algebra, its para-unit would span its commutative center, so it would coincide with the para-unit of $S\subo$, which is $e_1+e_2$. But,
\[
\begin{split}
(e_1+e_2)*v_1&=(e_1+e_2)\cdot\tau_{nst}^2(-v_1)=-\tau_{nst}^2(v_1)=v_2,\\
v_1*(e_1+e_2)&=\tau_{nst}(-v_1)\cdot(e_1+e_2)=-\tau_{nst}(v_1)=v_1-v_2,
\end{split}
\]
so that $\overline{C(k)}_{\tau_{nst}}$ has no para-unit. Hence $\overline{C(k)}_{\tau_{nst}}$ is an Okubo algebra. Moreover, the map $g:\overline{C(k)}_{\tau_{nst}}\rightarrow k$ such that $g(x)=n(x,x*x)$ for any $x$ satisfies (see \cite[\S 3]{Eld97}) that $g(x+y)=g(x)+g(y)$ and $g(\alpha x)=\alpha^3g(x)$ for any $x,y\in \overline{C(k)}_{\tau_{nst}}$ and $\alpha\in k$. Moreover, we have:
\[
g(e_1)=g(e_2)=1,\quad g(u_1)=g(u_2)=g(u_3)=0=g(v_1)=g(v_2)=g(v_3).
\]
Hence $g\left(\overline{C(k)}_{\tau_{nst}}\right)=k^3$ which, by \cite[Theorem 5.1]{Eld97}, forces $\overline{C(k)}_{\tau_{nst}}$ to be isomorphic to $P_8(k)$.
\end{proof}

\begin{remark}\label{re:gradingsP8ktaunst}
The automorphism $\tau_{nst}$ preserves the $5$-grading of the split Cayley algebra $C(k)$ in Theorem \ref{th:gradingsCayley}, and hence $P_8(k)$ inherits this grading. Besides, coarsening this grading we obtain the following group gradings $S=\oplus_{g\in G}S_g$ of the pseudo-octonion algebra $S=P_8(k)=\overline{C(k)}_{\tau_{nst}}$:
\begin{enumerate}
\item $G=\bZ$, ($5$-grading):\newline
 $S_0=\espan{e_1,e_2}$, $S_1=\espan{u_1,u_2}$, $S_2=\espan{v_3}$, $S_{-1}=\espan{v_1,v_2}$, $S_{-2}=\espan{u_3}$.

\smallskip
\item $G=\bZ_4$:\newline
    $S\subo=\espan{e_1,e_2}$, $S\subuno=\espan{u_1,u_2}$, $S_{\bar 2}=\espan{u_3,v_3}$, $S_{\bar 3}=\espan{v_1,v_2}$.

\smallskip
\item $G=\bZ_3$:\newline
 $S\subo=\espan{e_1,e_2},$ $S\subuno=\espan{u_1,u_2,u_3}$, $S_{\bar 2}=\espan{v_1,v_2,v_3}$.

\smallskip
\item $G=\bZ_2$:\newline
$S\subo=\espan{e_1,e_2,u_3,v_3}$ and $S\subuno=\espan{u_1,u_2,v_1,v_2}$. \hfill\qed
\end{enumerate}
\end{remark}

\medskip

A final presentation of $P_8(k)$ that will be used later on appears if the ground field $k$ is a field of characteristic $\ne 3$ containing a primitive cubic root $\omega$ of $1$. Consider then the automorphism $\tau_\omega$ of the split Cayley algebra $C(k)$ such that:
\begin{equation}\label{eq:tauomega}
\tau_\omega(e_i)=e_i\ (i=1,2),\quad \tau_\omega(u_i)=\omega^iu_i,\ \tau_\omega(v_i)=\omega^{-i}v_i\ (i=1,2,3).
\end{equation}

Then again the corresponding Petersson algebra is the pseudo-octonion algebra.

\begin{proposition}\label{pr:tauomega}
Let $k$ be a field of characteristic $\ne 3$ containing a primitive cubic root $\omega$ of $1$, then the Petersson algebra $\overline{C(k)}_{\tau_\omega}$ is isomorphic to the pseudo-octonion algebra.
\end{proposition}
\begin{proof} This can be checked directly, and it is also a direct consequence of \cite[Corollary 3.6]{EP96}.
\end{proof}

\begin{remark}\label{re:gradingsP8ktauomega}
The automorphism $\tau_\omega$ preserves the $\bZ^2$-grading of the split Cayley algebra $C(k)$ in Theorem \ref{th:gradingsCayley}, and hence $P_8(k)$ inherits this grading. Besides, coarsening this grading we obtain the following group gradings $S=\oplus_{g\in G}S_g$ of the pseudo-octonion algebra $S=P_8(k)=\overline{C(k)}_{\tau_\omega}$ (characteristic $\ne 3$):
\begin{enumerate}
\item $G=\bZ^2$:\newline
$S_{(0,0)}=\espan{e_1,e_2}$, $S_{(1,0)}=\espan{u}_1$, $S_{(0,1)}=\espan{u_2}$, $S_{(1,1)}=\espan{v_3}$, $S_{(-1,0)}=\espan{v_1}$, $S_{(0,-1)}=\espan{v_2}$, $S_{(-1,-1)}=\espan{u_3}$.

\smallskip
\item $G=\bZ\times \bZ_2$:\newline
  $S_{(0,\bar 0)}=\espan{e_1,e_2}$, $S_{(0,\bar 1)}=\espan{u_3,v_3}$, $S_{(1,\bar 0)}=\espan{u_1}$, $S_{(1,\bar 1)}=\espan{v_2}$, $S_{(-1,\bar 0)}=\espan{v_1}$, $S_{(-1,\bar 1)}=\espan{u_2}$.

\smallskip
\item $G=\bZ$ ($3$-grading):\newline
 $S_0=\espan{e_1,e_2,u_3,v_3}$, $S_1=\espan{u_1,v_2}$, $S_{-1}=\espan{u_2,v_1}$. \hfill\qed
\end{enumerate}
\end{remark}

\bigskip

\section{Group gradings on symmetric composition algebras}\label{se:GradSymCompo}

Let $(S,*,n)$ be a symmetric composition algebra, and assume that it is graded over a group $G$: $S=\oplus_{g\in G}S_g$, with $S_g*S_h\subseteq S_{gh}$ for any $g,h\in G$. As always, it will be assumed that $G$ is generated by those $g\in G$ with $S_g\ne 0$. Because of equation \eqref{eq:xyx}, for any $a,b,c\in G$ and nonzero elements $x\in S_a$, $y\in S_b$ and $z\in S_c$,
\[
(x*y)*z+(z*y)*x=n(x,z)y,
\]
so $n(S_a,S_c)=0$ unless either $abc=b$ or $cba=b$. With $b=a$, it follows that $n(S_a,S_c)=0$ unless $c=a^{-1}$. Because of the nondegeneracy of $n$, we may take elements $x\in S_a$ and $z\in S_{a^{-1}}$ such that $n(x,z)\ne 0$, and then we conclude that either $aba^{-1}=b$ or $a^{-1}ba=b$. In any case $ab=ba$ for any $a,b\in G$ such that $S_a\ne 0\ne S_b$, and since these elements generate $G$, it follows that, as for Hurwitz algebras, the grading group $G$ is abelian.

Therefore, in what follows, additive notation will be used for the grading groups.

\smallskip

The situation for gradings on para-Hurwitz algebras of dimension $4$ or $8$ is easily reduced to the Hurwitz situation:

\begin{lemma}\label{le:gradingspH}
Let $(S,\bullet,n)$ be a para-Hurwitz algebra of dimension $\geq 4$, and assume that $S=\oplus_{g\in G}S_g$ is graded over a group $G$. Then its para-unit belongs to $S_0$.
\end{lemma}
\begin{proof}
The commutative center of $S$, that is, the subspace $K(S)=\{x\in S: x\bullet y=y\bullet x\ \text{for any $y\in S$}\}$ is a graded subspace of $S$. But since the dimension is $\geq 4$, this center has dimension $1$ and it is spanned by the para-unit $e$. Thus, $e$ is homogeneous, so $e\in S_g$ for some $g\in G$. As $e$ is an idempotent, it follows that $g+g=g$, so $g=0$.
\end{proof}

\begin{theorem}\label{th:gradingspH}
Let $(C,\cdot,n)$ be a Hurwitz algebra of dimension $\geq 4$, and let $(\bar C,\bullet,n)$ be its associated para-Hurwitz algebra (so that $x\bullet y=\bar x\cdot \bar y$ for any $x,y\in C$). Then the  group gradings of $C$ and $\bar C$ coincide.
\end{theorem}
\begin{proof}
Let us denote by $e$ the unity of $C$, which is the para-unit of $\bar C$. It is clear that given a grading $C=\oplus_{g\in G}C_g$ of $C$, $e$ belongs to the subspace $C_0$, and hence given any $g\in G$ and $x\in C_g$, the element $\bar x=n(x,e)e-x$ belongs to $C_g$ too, as $n(S_g,S_0)=0$ unless $g=0$ (see \cite{GradingsOctonions}). Therefore, the grading is inherited by the para-Hurwitz algebra $\bar C$. The converse is equally trivial because of Lemma \ref{le:gradingspH}.
\end{proof}

Any four dimensional symmetric composition algebra is para-Hurwitz, while the eight dimensional symmetric composition algebras are either para-Hurwitz or Okubo. Hence we have to deal with the two dimensional symmetric composition algebras and the Okubo algebras.

\begin{theorem}\label{th:dimension2}
Let $(S,*,n)$ be a two dimensional symmetric composition algebra over a field $k$, and let $S=\oplus_{g\in G}S_g$ be a nontrivial group grading (that is $S\ne S_0$). Then either:
\begin{romanenumerate}
\item The characteristic of $k$ is $\ne 2$, $G=\bZ_2$, $S$ is a para-Hurwitz algebra with a para-unit $e\in S\subo$, $S\subo=ke$, and $S\subuno=S\subo^\perp$, or

\item The norm $n$ is isotropic and $G=\bZ_3$. In this case $S\subuno$ is spanned by an element $x$ with $n(x)=0\ne n(x,x*x)$, while $S_{\bar 2}$ is spanned by $x*x$.
\end{romanenumerate}
\end{theorem}
\begin{proof}
Since the grading is not trivial, there are two elements $g,h\in G$ such that $S=S_g\oplus S_h$, with $\dim S_g=\dim S_h=1$. Assume one of these elements, say $g$, is $0$. Then, because of the nondegeneracy of $n$ and since $n(S_g,S_h)=0$, it follows that the characteristic of $k$ is $\ne 2$ and $S_0=kx$ for some $x\in S$ with $n(x)\ne 0$. But then $0\ne x*x\in kx$, so a scalar multiple of $x$ is an idempotent $e$. Thus, $S_0=ke$, with $e*e=e$, so that $n(e)=1$. Take $0\ne z\in S_h$. Then $e*z=z*e=\alpha z$ for some $0\ne \alpha\in k$, and since $z=n(e)z=(e*z)*e=\alpha^2 z$, it follows that $\alpha=-1$, as $S$ is not unital. Hence $e$ is a para-unit of $S$ and the situation in (i) is obtained.

Assume, on the contrary, that $S_0=0$. Then, since $S*S=S$, we get $S_a*S_a= S_b$, $S_b*S_b=S_a$, and $S_a*S_b=S_b*S_a=0$. Thus, $2a=b$, $2b=a$, so $3a=0$ and $G$ is, up to isomorphism, the cyclic group of order $3$. The nondegeneracy of the norm gives (ii).
\end{proof}

Note that in the situation of Theorem \ref{th:dimension2}.(ii), if $x$ is a nonzero homogeneous element in $S\subuno$, $n(x)=0\ne n(x,x*x)$ (since the norm is nondegenerate). The multiplication in $(S,*,n)$ is completely determined in terms of $\alpha=n(x,x*x)$, as $x*(x*x)=(x*x)*x=n(x)x=0$, and $(x*x)*(x*x)=\alpha x$ (see the proof of Proposition \ref{pr:interesting}). Given an arbitrary element $e=\mu x+\nu x*x$, $e*e=\nu^2\alpha x+\mu^2x*x$, so $e$ is a nonzero idempotent if and only if $\mu^2=\nu$ and $\nu^2\alpha=\mu$. Thus, there are nonzero idempotents (and, as a consequence, the algebra is para-Hurwtiz) if and only if there exists an element $\epsilon \in k$ such that $\epsilon^3=\alpha$.

\medskip

We are left with the gradings on Okubo algebras.

\begin{theorem}\label{th:gradingsOkubo1}
Let $(S,*,n)$ be an Okubo algebra over a field $k$, and let $S=\oplus_{g\in G}S_g$ be a nontrivial group grading (that is $S\ne S_0$). Then either:
\begin{romanenumerate}
\item $0\ne S_0$ is a para-Hurwitz subalgebra, or
\item $S_0$ is a two-dimensional subalgebra without nonzero idempotents, $G=\bZ_3$, and there exist scalars $0\ne\alpha,\beta\in k$ and an isomorphism $\varphi:S\rightarrow \calO_{\alpha,\beta}$ such that
    \[
    \begin{split}
    \varphi(S\subo)&=\alg{x_{1,0}}=\espan{x_{1,0},x_{-1,0}},\\
    \varphi(S\subuno)&=\espan{x_{0,1},x_{1,1},x_{-1,1}},\\
    \varphi(S_{\bar 2})&=\espan{x_{0,-1},x_{1,-1},x_{-1,-1}},
    \end{split}
    \]
    that is, up to isomorphism, the grading is the standard $\bZ_3$-grading of the Okubo algebra $\calO_{\alpha,\beta}$ (see \eqref{eq:Z3standard} in Remark \ref{re:Z3naturalgrading}),
    or
\item $S_0=0$, $G=\bZ_3^2$ and there exist scalars $0\ne \alpha,\beta\in k$ and an isomorphism $\varphi:S\rightarrow \calO_{\alpha,\beta}$ of graded algebras, with the standard $\bZ_3^2$-grading in $\calO_{\alpha,\beta}$ (see  \eqref{eq:Z32standard} in Remark \ref{re:Z3naturalgrading}).
\end{romanenumerate}
\end{theorem}
\begin{proof}
Since it is assumed that the grading is not trivial, $S_0$ is either $0$ or a proper composition subalgebra of $S$. Assume first that $S_0$ is not zero. Then either $S_0$ is a para-Hurwitz subalgebra, so the situation in (i) holds, or $S_0$ is a two dimensional symmetric composition algebra without nonzero idempotents.

In the latter case, let $0\ne g\in G$ with $S_g\ne 0$, and consider the subalgebra $\oplus_{n\in\bZ}S_{ng}$. If this subalgebra is not the whole $S$, then it has dimension $4$, and hence it is para-Hurwitz. But then its para-unit belongs to the zero component $S_0$ (Lemma \ref{le:gradingspH}), a contradiction. Hence $S=\oplus_{n\in\bZ}S_{ng}$ and $G$ is cyclic and generated by any element $g\in G$ with $S_g\ne 0$. Extend scalars up to an algebraic closure $\hat k$ of $k$. Then $\hat S=\hat k\otimes_k S$ is a symmetric composition algebra over $\hat k$ graded over $G$, and $\hat S_0$ is a two dimensional symmetric composition algebra over the algebraically closed field $\hat k$, so that $\hat S_0$ is a para-Hurwitz algebra. It does no harm to denote the norm in $\hat S$ also by $n$. Let $e$ be a para-unit of $\hat S_0$. Consider the Cayley algebra $(\hat S,\cdot,n)$ defined by means of
\[
x\cdot y=(e*x)*(y*e),
\]
whose unity is $e$. Since $e\in \hat S_0$, this Cayley algebra inherits the grading on $\hat S$. Thus we have a grading on a Cayley algebra over a cyclic group $G$, with two dimensional zero part and with the property that $G$ is generated by any nonzero $g\in G$ with $\hat S_g\ne 0$. A careful look at Theorem \ref{th:gradingsCayley} shows that the only possibility for $G$ is $\bZ_3$. Hence $S=S\subo\oplus S\subuno\oplus S_{\bar 2}$ and $\dim S\subuno=\dim S_{\bar 2}=3$. Now,
\[
n(S\subuno*S_{\bar 2},S\subo)=n(S\subuno,S_{\bar 2}*S\subo)=n(S\subuno,S_{\bar 2})\ne 0
\]
as the norm is nondegenerate, so there exist elements $a\in S\subuno$ and $b\in S_{\bar 2}$ such that $0\ne a*b\in S\subo$. But $n(S\subuno)=0=n(S_{\bar 2})$ because of \eqref{eq:xyx}, so $n(a*b)=0$. We conclude that the restriction of the norm to $S\subo$ is isotropic. Let $0\ne x\in S\subo$ be an element with $n(x)=0$. Since $S\subo$ is a form of a two dimensional para-Hurwitz algebra, it contains no element whose square is $0$. Since $S\subo$ has no nonzero idempotents, it follows that $\{x,x*x\}$ is a basis of $S\subo$. Let $\alpha=n(x,x*x)$, which is nonzero as the restriction of $n$ to $S\subo$ is nondegenerate. The argument in the paragraph previous to the Theorem shows that actually $\alpha\not\in k^3$.

Consider, as in the proof of Proposition \ref{pr:interesting}, the element $p=x+\alpha^{-1}x*x$ and the Cayley algebra defined on $S$ with multiplication \[
a\cdot b=l_p(a)*r_p(b),
\]
whose unity is the element $q=p*p=\alpha^{-1}x+x*x$. Also, $e_1=\alpha^{-1}x$ and $e_2=x*x$ are idempotents of $C$ whose sum is the unity. Besides $(S,\cdot,n)$ inherits the grading from $S$, as $p\in S\subo$, and hence (Theorem \ref{th:gradingsCayley}) the $\bZ_3$-grading is given by the Peirce decomposition relative to $e_1$ and $e_2$. (Note that $e_1$ and $e_2$ are the only nonzero isotropic idempotents in $(S\subo,\cdot)$.) Thus it follows that either
\[
S\subuno=\{z\in S: e_1\cdot z=z\cdot e_2=z\},\quad
S_{\bar 2}=\{ z\in S: e_2\cdot z=z\cdot e_1=z\},
\]
or
\[
S\subuno =\{ z\in S: e_2\cdot z=z\cdot e_1=z\},\quad
S_{\bar 2}=\{z\in S: e_1\cdot z=z\cdot e_2=z\}.
\]
Assume, for instance, that the first situation happens and take $0\ne y\in S\subuno$. The arguments in the proof of Proposition \ref{pr:interesting} show that the elements $x$ and $y$ satisfy the hypotheses of Corollary \ref{co:bonito} and we obtain the result in (ii), with $\alpha=n(x,x*x)$ and $\beta=n(y,y*y)$. If it is the second situation above the one that happens, one gets $y*x=0$ instead of $x*y=0$, but one can change $x$ by $x*x$, which then satisfies $(x*x)*y=-(y*x)*x=0$ and get the same conclusion.

\smallskip

Finally, assume that $S_0=0$. Take any element $g\in G$ such that $S_g\ne 0$. If the order of $g$ is $2$, then $S_g*S_g\subseteq S_0=0$, so $n(S_g)=0$ and $n(S_g,S_h)=0$ for any $h\ne g^{-1}=g$, which is a contradiction with the nondegeneracy of the norm. Hence the order of any $g\in G$ with $S_g\ne 0$ is at least $3$. Take an element $g\in G$ with $S_g\ne 0$ and consider the subalgebra $\oplus_{n\in\bZ}S_{ng}$. This is a composition subalgebra of $S$, and it cannot be a four dimensional subalgebra, as this would imply this subalgebra to be para-Hurwitz, with a unique para-unit which necessarily belongs to $S_0$ according to Lemma \ref{le:gradingspH}. Thus either this subalgebra is the whole $S$, or it has dimension $2$. In the latter case, and because of Theorem \ref{th:dimension2}, the order of $g$ is exactly $3$ and $\dim S_g=1$. Therefore, by dimension count, either $G=\bZ_3^2$ with $\dim S_g=1$ for any $g\ne 0$, or $G$ is cyclic and generated by an element $g\in G$ with $S_g\ne 0$.

In the first case ($G=\bZ_3^2$), take $0\ne x\in S_{(\bar 1,\bar 0)}$, then $n(x)=0$ (as $(x*x)*x=n(x)x\in S_0=0$). Since $S_{(\bar 1,\bar 0)}\oplus S_{(\bar 2,\bar 0)}=S_{(\bar 0,\bar 0)}\oplus S_{(\bar 1,\bar 0)}\oplus S_{(\bar 2,\bar 0)}$ is a two dimensional composition subalgebra of $S$, we can conclude as before that $x*x$ is a nonzero element in $S_{(\bar 2,\bar 0)}$, and that $\alg{x}=\espan{x,x*x}=S_{(\bar 1,\bar 0)}\oplus S_{(\bar 2,\bar 0)}$. Similarly, if $0\ne y\in S_{(\bar 0,\bar 1)}$, $n(y)=0$, and $\alg{y}=\espan{y,y*y}=S_{(\bar 0,\bar 1)}\oplus S_{(\bar 0,\bar 2)}$. Theorem \ref{th:bonito} shows that either $x*y=0$ or $y*x=0$. In the first case Corollary \ref{co:bonito} shows that with $\alpha=n(x,x*x)$ and $\beta=n(y,y*y)$, there is an isomorphism $\varphi:S\rightarrow \calO_{\alpha,\beta}$ such that $\varphi(x)=-x_{1,0}$ and $\varphi(y)=-x_{0,1}$, and this gives the required graded isomorphism. Otherwise just permute the roles of $x$ and $y$.

On the other hand, if $G$ is cyclic generated by the element $g$ with $S_g\ne 0$, take an arbitrary nonzero element  $x\in S_g$, and then take an element $y\in S_{-g}$ with $n(x,y)=1$ (the norm is nondegenerate). Since $n(x)=0=n(y)$ because $S_0=0$, $n(x+y)=n(x,y)=1$, so that the left multiplication by $x+y$: $l_{x+y}$, is an isometry. Then $x*x=(x+y)*x\ne 0$ (note that $y*x\in S_0=0$). Therefore the square of any nonzero element in $S_g$ is $\ne 0$. In the same vein, $0\ne (x*x)*(x*x)=-((x*x)*x)*x+n(x,x*x)x=n(x,x*x)x$, so that $n(x,x*x)\ne 0$,  $x*x\in S_{-g}$ and $2g=-g$, so $3g=0$ and $G=\bZ_3$. But then $\dim S_g=4=\dim S_{-g}$. Take $z\in S_{-g}$ with $n(x,z)=0$. Then $(x*x)*z+(z*x)*x=n(x,z)x=0$, and $z*x\in S_0=0$. Thus $(x+x*x)*z=0$, although $n(x+x*x)=n(x,x*x)\ne 0$, so $l_{x+x*x}$ is an isometry, and in particular a bijection. Thus a contradiction is reached, and this finishes the proof.
\end{proof}

\smallskip

Therefore, in order to determine the gradings of the Okubo algebras, we must consider the gradings where $S_0$ is a para-Hurwitz algebra. Let $e$ be a para-unit of $S_0$. Moreover, if for a subgroup $H$ of $G$, the subalgebra $\oplus_{g\in H}S_g$ has dimension $4$, and hence it is a para-quaternion algebra, then its unique para-unit is in $S_0$ (Lemma \ref{le:gradingspH}), so $e$ can be taken to be this para-unit.

Thus, let $(S,*,n)$ be an Okubo algebra, with a grading $S=\oplus_{g\in G}S_g$ such that $S_0$ is para-Hurwitz with para-unit $e$ as above. Consider then the Cayley algebra $(S,\cdot,n)$ with
\[
x\cdot y=(e*x)*(x*e),
\]
whose unity is $e$. Besides, the linear map $\tau:S\rightarrow S$ such that $\tau(x)=n(e,x)e-x*e=\bar x*e$ (where $\bar x=n(e,x)e-x$) is an automorphism of both $(S,*,n)$ and $(S,\cdot,n)$ with $\tau^3=1$, and such that the multiplication $*$ is given by the equation:
\[
x*y=\tau(\bar x)\cdot \tau^2(\bar y)
\]
(see \cite[Theorem 2.5]{EP96}).

As $(S,*,n)$ is not para-Hurwitz, $\tau\ne 1$ holds, so the order of $\tau$ is exactly $3$. Besides, $e$ is a para-unit of $S_0$, so the restriction of $\tau$ to $S_0$ is the identity map. Also, since $e\in S_0$, $\tau(S_g)=S_g$ for any $g\in G$, and $\bar S_g=S_g$ too. Therefore, $\oplus_{g\in G}S_g$ is a grading of the Cayley algebra $(S,\cdot ,n)$ too, where all the homogeneous spaces are invariant under the automorphism $\tau$.

Conversely, given any grading $S=\oplus_{g\in G}S_g$ of the Cayley algebra $(S,\cdot,n)$ such that $\tau(S_g)=S_g$ for any $g\in G$ and with $e\in S_0$ ($e$ is the unity of $(S,\cdot,n)$), this is a grading of $(S,*,n)$ too. Moreover, the universal grading group does not depend on consider this grading as a grading of $(S,\cdot,n)$ or of $(S,*,n)$, since $S_{g_1}*S_{g_2}=\tau(\bar S_{g_1})\cdot\tau^2(\bar S_{g_2})=S_{g_1}\cdot S_{g_2}$ for any $g_1,g_2\in G$.

\smallskip

Hence we must look at the possible gradings of a Cayley algebra $(C,\cdot,n)$ such that there is an automorphism $\tau$ of order $3$ which leaves invariant all the homogeneous spaces, whose restriction to the homogeneous subspace of degree $0$ is the identity, and such that the Petersson algebra $\bar C_\tau$ is an Okubo algebra. In this way all the possible gradings of Okubo algebras with $S_0$ being para-Hurwitz will be obtained. Let us do this by reviewing all the possibilities in Theorem \ref{th:gradingsCayley}. The notations in the paragraphs above will be kept throughout the discussion.

\medskip

\noindent\textbf{1.} $G=\bZ_2$:

If the characteristic of the ground field $k$ is $\ne 3$, the proof of Proposition \ref{pr:idempotentsno3} shows that $S\subo=\{x\in S:\tau(x)=x\}$, and $S\subuno$ is the orthogonal complement to $S\subo$. Hence the grading is given by Remark \ref{re:Z2gradingsOkubo}.

On the other hand, if the characteristic is $3$, with $Q=S\subo$, $S=Q\oplus Q\cdot u$, and $\tau(u)=w\cdot u$ for some $w\in Q$. Since $\tau^3=1\ne \tau$, $w^{\cdot 3}=1\ne w$, so $(w-1)^{\cdot 3}=0\ne w-1$. But $(Q,\cdot)$ is a quaternion algebra, hence of degree $2$, so that $(w-1)^{\cdot 2}=1$. In particular, $Q$ has a nonzero nilpotent element, and thus it is isomorphic to $\Mat_2(k)$. Hence the restriction of the norm to $Q$ represents any scalar, and this shows that we can take the element $u$ orthogonal to $Q$, of norm $1$. Moreover, without loss of generality, we can assume that $Q=\Mat_2(k)$ and that $w=\left(\begin{smallmatrix} 1&1\\ 0&1\end{smallmatrix}\right)$. Thus $(S,*,n)$ is uniquely determined as $\bar C_\tau$, where $C=CD(\Mat_2(k),-1)$ (the split Cayley algebra), and $\tau$ given by $\tau(x)=x$ for any $x\in Q$, and $\tau(u)=w\cdot u$. Hence, by uniqueness, this is the situation given in the $\bZ_2$-grading in Remark \ref{re:gradingsP8ktaunst}.

\bigskip

\noindent\textbf{2.} $G=\bZ_2^2$:

Since $S_{(\bar 0,\bar 0)}\oplus S_{(\bar 1,\bar 0)}$ has dimension $4$, as mentioned above, the para-unit $e$ of $S_{(\bar 0,\bar 0)}$ used to define $\tau$ will be taken to be the para-unit of this para-quaternion algebra. Then, if the characteristic of $k$ is $\ne 3$, we are in the situation of the $\bZ_2^2$-grading in Remark \ref{re:Z2gradingsOkubo}.

On the other hand, if the characteristic of $k$ is $3$, $(S_0,\cdot,n)$ is a two dimensional Hurwitz algebra, and hence a commutative associative separable algebra. Besides, the homogeneous components are orthogonal subspaces. For any of the other homogeneous components, take a non isotropic element $u$. Then $\tau(u)=w\cdot u$ for some $w\in S_0$. As in the previous case, the element $w$ satisfies $(w-1)^{\cdot 2}=0$, but $(S_0,\cdot)$ has no nilpotent elements. Hence $w=1$. But this forces $\tau$ to be the identity, a contradiction. Therefore this situation does not appear in characteristic $3$.

\bigskip

\noindent\textbf{3.} $G=\bZ_2^3$ and the characteristic of $k$ is $\ne 2$:

In this case, $S_0$ is the one-dimensional Hurwitz algebra (the ground field) and the homogeneous components are orthogonal subspaces. As before, take a non isotropic element $u$ in any of the other homogeneous components, then $\tau(u)=w\cdot u$ for some $w\in S_0$. Since $u\cdot u\in S_0$, $w\cdot w=1$, and since $\tau^3=1$, $w^{\cdot 3}=1$. It follows that $w=1$, but this gives $\tau=1$, a contradiction. Thus, this situation is not possible.

\bigskip

\noindent\textbf{4.} $G=\bZ_3$.

Here there is a canonical basis of the Cayley algebra $(S,\cdot,n)$ such that $S\subo=\espan{e_1,e_2}$, $S\subuno=\espan{u_1,u_2,u_3}$ and $S_{\bar 2}=\espan{v_1,v_2,v_3}$ (see Theorem \ref{th:gradingsCayley}). If the characteristic of $k$ is $\ne 3$, and because of \cite[Theorem 3.6]{EP96}, the subalgebra fixed by $\tau$: $\{x\in S:\tau(x)=x\}$, is four dimensional and contains $S\subo$. Hence $1$ is an eigenvalue of multiplicity $1$ of the restriction of $\tau$ to the Peirce spaces $U=\espan{u_1,u_2,u_3}$ and $V=\espan{v_1,v_2,v_3}$. Since $\tau$, as an automorphism of the Cayley algebra $(S,\cdot,n)$, is an isometry of order $3$, and the subspaces $U$ and $V$ are isotropic and paired by the polar form of the norm, the minimal polynomial of the restriction of $\tau$ to both $U$ and $V$ is $X^3-1$. Thus, there is an element $u\in U$ such that $\{u,\tau(u),\tau^2(u)\}$ is a basis of $U$. Take the isotropic elements $x=e_1$ and $y=u$. Then:
\[
\begin{split}
x*x&=\bar e_1\cdot\bar e_1=e_2\cdot e_2=e_2,\\
y*y&=\tau(u)\cdot\tau^2(u)\in V\setminus\{0\},\\
x*y&=-\bar e_1\cdot\tau^2(u)=-e_2\cdot\tau^2(u)=0,
\end{split}
\]
so that $\alg{x}$ and $\alg{y}$ are orthogonal two dimensional composition subalgebra, and hence, due to Corollary \ref{co:bonito}, $(S,*,n)$ is isomorphic to the Okubo algebra $\calO_{1,\beta}$ with $\beta=n(y,y*y)=n(u,\tau(u)\cdot\tau^2(u))$ (note that $n(x,x*x)=n(e_1,e_2)=1$), through an isomorphism that takes $x$ to $-x_{1,0}$ and $y$ to $-x_{0,1}$. We conclude that, up to isomorphism, the grading on $(S,*,n)$ is the standard $\bZ_3$-grading in $\calO_{1,\beta}$ (see Remark \ref{re:Z3naturalgrading}).

On the other hand, if the characteristic of $k$ is $3$, we merely have that the minimal polynomial of the restriction of $\tau$ to either $U$ or $V$ divides $X^3-1=(X-1)^3$. Since $\tau\ne 1$, either this minimal polynomial is $X^3-1$, and then the argument for characteristic $\ne 3$ works equally well here, or this minimal polynomial is $(X-1)^2$. In the latter case, a canonical basis $\{u_1,u_2,u_3\}$ of $U$ (and the corresponding dual basis of $V$) can be taken so that $\tau$ becomes the automorphism $\tau_{nst}$ in \eqref{eq:taunst}. Thus our Okubo algebra is the Petersson algebra $\overline{C(k)}_{\tau_{nst}}$, which is isomorphic to the pseudo-octonion algebra (Proposition \ref{pr:taunst}), and our grading is the $\bZ_3$-grading in Remark \ref{re:gradingsP8ktaunst}. This grading is not equivalent to any standard $\bZ_3$-grading on an Okubo algebra $\calO_{\alpha,\beta}$. The reason is that the only nonzero idempotent in our $S\subo$ is $e=e_1+e_2$ which satisfies that the dimension of the subspace $\{x\in S: e*x=x*e\}$ is $6$ (it coincides with the subspace of elements fixed by $\tau_{nst}$). However, in $(\calO_{\alpha,\beta})\subo$ (with the standard $\bZ_3$-grading), either there is no nonzero idempotent, or $\alpha\in k^3$ and the only nonzero idempotent is $-\alpha^{-\frac{1}{3}}x_{1,0}-\alpha^{\frac{1}{3}}x_{-1,0}$, and this idempotent satisfies that the subspace of elements that commute with it has dimension $4$.

\bigskip

\noindent\textbf{5.} $G=\bZ_4$.

In this case, $S\subo\oplus S_{\bar 2}$ is a four dimensional subalgebra, and hence the para-unit $e$ of $S\subo$ can be taken to be the para-unit of this subalgebra. Thus, our Okubo algebra $(S,*,n)$ is a Petersson algebra with the automorphism $\tau$ fixing elementwise the subspace $S\subo\oplus S_{\bar 2}$. According to Theorem \ref{th:gradingsCayley}, the Cayley algebra involved is split and has a canonical basis such that $S\subo=\espan{e_1,e_2}$, $S\subuno=\espan{u_1,u_2}$, $S_{\bar 2}=\espan{u_3,v_3}$ and $S_{\bar 3}=\espan{v_1,v_2}$. If the characteristic of $k$ is $\ne 3$, the subalgebra of elements fixed by $\tau$ has dimension $4$ and hence coincides with $S\subo\oplus S_{\bar 2}$, and the minimal polynomial of the restriction of $\tau$ to $S\subuno$ or its dual subspace $S_{\bar 3}$ is $X^2+X+1$. On the other hand, if the characteristic of $k$ is $3$, the minimal polynomial of these restrictions divides $X^3-1=(X-1)^3$ and it is not $X-1$. By dimension count, this minimal polynomial is $(X-1)^2=X^2+X+1$. So no matter which characteristic we are dealing with, the minimal polynomial of the restriction of $\tau$ to $S\subuno$ and to $S_{\bar 3}$ is $X^2+X+1$. The canonical basis can be adjusted to assume that $\tau$ coincides with the automorphism $\tau_{nst}$ in \eqref{eq:taunst}, and hence our Okubo algebra is the pseudo-octonion algebra (Proposition \ref{pr:taunst}), and our grading is the $\bZ_4$-grading in Remark \ref{re:gradingsP8ktaunst}.

\bigskip

\noindent\textbf{6.} $G=\bZ$ ($3$-grading).

According to Theorem \ref{th:gradingsCayley}, our Cayley algebra is split and there is a canonical basis such that $S_0=\espan{e_1,e_2,u_3,v_3}$, $S_1=\espan{u_1,v_2}$ and $S_{-1}=\espan{u_2,v_1}$. The order $3$ automorphism $\tau$ fixes elementwise $S_0$, and hence $\tau(u_1)=\tau(e_1\cdot u_1)=e_1\cdot\tau(u_1)\in e_1\cdot S_1=\espan{u_1}$. Thus there is a scalar $\alpha\in k$ such that $\tau(u_1)=\alpha u_1$. Since $\tau^3=1$, $\alpha^3=1$. If the characteristic of $k$ is $3$ or it is $\ne 3$ but $k$ does not contain the primitive cubic roots of $1$, the scalar $\alpha$ equals $1$. In the same vein, $\tau$ fixes any of the elements $u_2$, $v_1$ and $v_2$ and is thus the identity, a contradiction.

Therefore, this possibility may happen only if $k$ is a field of characteristic not $3$ containing the three cubic roots of $1$. In this case, the subspace of elements fixed by $\tau$ is $S_0$ (\cite[Theorem 3.6]{EP96}) so there is a primitive cubic root $\omega$ of $1$ in $k$ such that $\tau(u_1)=\omega u_1$, $\tau(v_2)=\tau(u_3\cdot u_1)=\omega v_2$, $\tau(u_2)=\omega^2 u_2$ (as $v_3=\tau(v_3)=\tau(u_1\cdot u_2)$) and $\tau(v_1)=\omega^2 v_1$. That is, our automorphism $\tau$ is the automorphism $\tau_\omega$ in \eqref{eq:tauomega}, our Okubo algebra is the pseudo-octonion algebra (Proposition \ref{pr:tauomega}), and our grading is the $3$-grading that appears in Remark \ref{re:gradingsP8ktauomega}.

\bigskip

\noindent\textbf{7.} $G=\bZ$ ($5$-grading).

Here the subalgebra $S=S_0\oplus S_2\oplus S_{-2}$ has dimension $4$, and hence the para-unit $e$ can be taken to be the unique para-unit of this subalgebra. Again the Cayley algebra here is the split Cayley algebra with a canonical basis (Theorem \ref{th:gradingsCayley}) such that $S_0=\espan{e_1,e_2}$, $S_1=\espan{u_1,u_2}$, $S_2=\espan{v_3}$, $S_{-1}=\espan{v_1,v_2}$ and $S_{-2}=\espan{u_3}$. The situation is completely similar to the case of a $\bZ_4$-grading. (Actually, from any $5$-grading one obtains a $\bZ_4$-grading by reducing modulo $4$.) The canonical basis can be adjusted to assume that $\tau$ coincides with the automorphism $\tau_{nst}$ in \eqref{eq:taunst}, and hence our Okubo algebra is the pseudo-octonion algebra (Proposition \ref{pr:taunst}), and our grading is the $5$-grading in Remark \ref{re:gradingsP8ktaunst}.

\bigskip

\noindent\textbf{8.} $G=\bZ^2$.

In this case the subalgebra $S_{(0,0)}\oplus S_{(1,1)}\oplus S_{(-1,-1)}$ has dimension four, and hence the para-unit $e$ can be taken to be the unique para-unit of this subalgebra. Accordingly, our automorphism $\tau$ fixes elementwise this subalgebra. Our Cayley algebra is split and there is a canonical basis such that (Theorem \ref{th:gradingsCayley}) $S_{(0,0)}=\espan{e_1,e_2}$, $S_{(1,0)}=\espan{u_1}$, $S_{(0,1)}=\espan{u_2}$, $S_{(1,1)}=\espan{v_3}$, $S_{(-1,0)}=\espan{v_1}$, $S_{(0,-1)}=\espan{v_2}$ and $S_{(-1,-1)}=\espan{u_3}$. As in the case of $3$-grading, this possibility can only occur if $k$ is a field of characteristic $\ne 3$ containing the primitive cubic roots of $1$. A primitive cubic root $\omega$ of $1$ can be taken so that $\tau$ is the automorphism $\tau_\omega$ in \eqref{eq:tauomega}, our Okubo algebra is the pseudo-octonion algebra (Proposition \ref{pr:tauomega}), and our grading is the $\bZ^2$-grading that appears in Remark \ref{re:gradingsP8ktauomega}.

\bigskip

\noindent\textbf{9.} $G=\bZ\times\bZ_2$.

This case is completely similar to the previous one. It may occur only if $k$ is a field of characteristic $\ne 3$ containing the primitive cubic roots of $1$. A primitive cubic root $\omega$ of $1$ can then be taken so that $\tau$ is the automorphism $\tau_\omega$ in \eqref{eq:tauomega}, our Okubo algebra is the pseudo-octonion algebra (Proposition \ref{pr:tauomega}), and our grading is the $\bZ\times\bZ_2$-grading that appears in Remark \ref{re:gradingsP8ktauomega}.

\bigskip

Our next result summarizes all the work done in this Section, and provides a complete description of all the nontrivial gradings of the symmetric composition algebras.

\begin{theorem}\label{th:gradingssymmetric}
Let $S=\oplus_{g\in G}S_g$ be a nontrivial grading of the symmetric composition algebra $(S,*,n)$ over a field $k$, where $G$ is the universal grading group. Then either:
\begin{itemize}
\item $(S,*,n)$ is the para-Hurwitz algebra attached to the Hurwitz algebra $(S,\cdot,n)$ (so that $x*y=\bar x\cdot\bar y$ for any $x,y\in S$), and the grading is given by a grading of the Hurwitz algebra $(S,\cdot,n)$. (See Theorem \ref{th:gradingsCayley}, Corollary \ref{co:gradingsquaternion} and the subsequent paragraph.)

\smallskip

\item The dimension of $S$ is $2$, the norm $n$ is isotropic, $G=\bZ_3$ and there is an element $x\in S$ with $n(x)=0\ne n(x,x*x)$ such that $S\subo=0$, $S\subuno$ is spanned by $x$ and $S_{\bar 2}$ is spanned by $x*x$. (See Theorem \ref{th:dimension2}.)

\smallskip

\item $(S,*,n)$ is an Okubo algebra and either:
\smallskip
\begin{enumerate}
\item $G=\bZ_2$:\newline
  If the characteristic of $k$ is $\ne 3$, there is a quaternion algebra $Q$ containing the two dimensional subalgebra $K=k1+kw$, with $w^2+w+1=0$, and there is a nonzero scalar $\alpha\in k$, such that $(S,*,n)$ is, up to isomorphism, the Petersson algebra $\bar C_\tau$, where $C$ is the Cayley-Dickson algebra $C=CD(Q,\alpha)=Q\oplus Qu$, $u^2=\alpha$, and $\tau$ is the order $3$ automorphism of $Q$ determined by $\tau(q)=q$ for any $q\in Q$ and $\tau(u)=wu$. Moreover, $S\subo=Q$ and $S\subuno=Qu$. \newline
  If the characteristic of $k$ is $3$, then $(S,*,n)$ is isomorphic, as a graded algebra, to the pseudo-octonion algebra $P_8(k)$, viewed as the Petersson algebra $\overline{C(k)}_{\tau_{nst}}$ as in Proposition \ref{pr:taunst}, with the grading given by
  \[
  S\subo=\espan{e_1,e_2,u_3,v_3},\quad S\subuno=\espan{u_1,u_2,v_1,v_2}.
  \]

\smallskip
\item $G=\bZ_2^2$ (characteristic $\ne 3$):\newline
   There is a quaternion algebra $Q$ containing the two dimensional subalgebra $K=k1+kw$, with $w^2+w+1=0$, and there is a nonzero scalar $\alpha\in k$, such that $(S,*,n)$ is, up to isomorphism, the Petersson algebra $\bar C_\tau$, where $C$ is the Cayley-Dickson algebra $C=CD(Q,\alpha)=Q\oplus Qu$, $u^2=\alpha$, and $\tau$ is the order $3$ automorphism of $Q$ determined by $\tau(q)=q$ for any $q\in Q$ and $\tau(u)=wu$. Moreover, $S_{(\bar 0,\bar 0)}=K$, $S_{(\bar 1,\bar 0)}=K^\perp\cap Q$, $S_{(\bar 0,\bar 1)}=Ku$ and $S_{(\bar 1,\bar 1)}=(K^\perp\cap Q)u$.

\smallskip
\item $G=\bZ_3$ (standard):\newline
  There are nonzero scalars $\alpha,\beta\in k$ such that $(S,*,n)$ is isomorphic, as a graded algebra, to the Okubo algebra $\calO_{\alpha,\beta}$ with its standard $\bZ_3$-grading in \eqref{eq:Z3standard}.

\smallskip
\item $G=\bZ_3$ (nonstandard, characteristic $3$):\newline
  The characteristic of $k$ is $3$, and $(S,*,n)$ is isomorphic, as a graded algebra, to the pseudo-octonion algebra $P_8(k)$, viewed as the Petersson algebra $\overline{C(k)}_{\tau_{nst}}$, as in Proposition \ref{pr:taunst}, with the grading given by:
  \[
  \null\qquad\qquad S\subo=\espan{e_1,e_2},\ S\subuno=\espan{u_1,u_2,u_3},\ S_{\bar 2}=\espan{v_1,v_2,v_3}.
  \]

\smallskip
\item $G=\bZ_3^2$:\newline
 There are nonzero scalars $\alpha,\beta\in k$ such that $(S,*,n)$ is isomorphic, as a graded algebra, to the Okubo algebra $\calO_{\alpha,\beta}$ with its standard $\bZ_3^2$-grading in \eqref{eq:Z32standard}.

\smallskip
\item $G=\bZ_4$:\newline
   $(S,*,n)$ is isomorphic, as a graded algebra, to the pseudo-octonion algebra $P_8(k)$, viewed as the Petersson algebra $\overline{C(k)}_{\tau_{nst}}$, as in Proposition \ref{pr:taunst}, with the grading given by:
  \[
  \begin{aligned}
  S\subo&=\espan{e_1,e_2},&\quad S\subuno&=\espan{u_1,u_2},\\
  S_{\bar 2}&=\espan{u_3,v_3},& S_{\bar 3}&=\espan{v_1,v_2}.
  \end{aligned}
  \]

\smallskip
\item $G=\bZ$ ($3$-grading, characteristic $\ne 3$):\newline
   The ground field $k$ contains the primitive cubic roots of $1$ and $(S,*,n)$ is isomorphic, as a graded algebra, to the pseudo-octonion algebra $P_8(k)$, viewed as the Petersson algebra $\overline{C(k)}_{\tau_\omega}$, as in Proposition \ref{pr:tauomega}, with the grading given by:
   \[
   \null\qquad\qquad S_0=\espan{e_1,e_2,u_3,v_3},\ S_1=\espan{u_1,v_2},\ S_{-1}=\espan{u_2,v_1}.
   \]

\smallskip
\item $G=\bZ$ ($5$-grading):\newline
  $(S,*,n)$ is isomorphic, as a graded algebra, to the pseudo-octonion algebra $P_8(k)$, viewed as the Petersson algebra $\overline{C(k)}_{\tau_{nst}}$, as in Proposition \ref{pr:taunst}, with the grading given by:
  \[
  \begin{aligned}
  \null\qquad\qquad S_0&=\espan{e_1,e_2},&\ S_1&=\espan{u_1,u_2},&\ S_ 2&=\espan{v_3},\\ &&S_{-1}&=\espan{v_1,v_2},&\ S_{-2}&=\espan{u_3}.
  \end{aligned}
  \]

\smallskip
\item $G=\bZ^2$ (characteristic $\ne 3$): \newline
  The ground field $k$ contains the primitive cubic roots of $1$ and $(S,*,n)$ is isomorphic, as a graded algebra, to the pseudo-octonion algebra $P_8(k)$, viewed as the Petersson algebra $\overline{C(k)}_{\tau_\omega}$, as in Proposition \ref{pr:tauomega}, with the grading given by:
   \[
   \begin{aligned}
    \null\qquad&& S_{(0,0)}&=\espan{e_1,e_2},&&\\
     S_{(1,0)}&=\espan{u_1},& S_{(0,1)}&=\espan{u_2},&\  S_{(1,1)}&=\espan{v_3},\\
     S_{(-1,0)}&=\espan{v_1},& S_{(0,-1)}&=\espan{v_2},& S_{(-1,-1)}&=\espan{u_3}.
    \end{aligned}
    \]

\smallskip
\item $G=\bZ\times\bZ_2$ (characteristic $\ne 3$):\newline
  The ground field $k$ contains the primitive cubic roots of $1$ and $(S,*,n)$ is isomorphic, as a graded algebra, to the pseudo-octonion algebra $P_8(k)$, viewed as the Petersson algebra $\overline{C(k)}_{\tau_\omega}$, as in Proposition \ref{pr:tauomega}, with the grading given by:
   \[
   \begin{aligned}
    S_{(0,\bar 0)}&=\espan{e_1,e_2},&\quad S_{(0,\bar 1)}&=\espan{u_3,v_3},\\
     S_{(1,\bar 0)}&=\espan{u_1},& S_{(1,\bar 1)}&=\espan{v_2},\\ \null\qquad S_{(-1,\bar 0)}&=\espan{v_1},&
     S_{(-1,\bar 1)}&=\espan{u_2}.
    \end{aligned}
   \]
\end{enumerate}
\end{itemize}
\end{theorem}

\medskip

\begin{remark}\label{re:coarsenings}
Over an algebraically closed field $k$ of characteristic $\ne 3$, all the gradings of the pseudo-octonion algebra in Theorem \ref{th:gradingssymmetric} are, up to isomorphism, coarsenings of either the $\bZ^2$-grading or the $\bZ_3^2$-grading.

Actually, both the $3$-grading and the $\bZ\times\bZ_2$-grading  are clearly coarsenings of the $\bZ^2$-grading, while the $\bZ_4$-grading is a coarsening of the $5$-grading, the standard $\bZ_3$-grading is a coarsening of the $\bZ_3^2$-grading, the $\bZ_2$-grading is a coarsening of the $\bZ_2^2$-grading, and the nonstandard $\bZ_3$-grading does not appear in characteristic $\ne 3$. Hence, it is enough to check that both the $5$-grading and the $\bZ_2$-grading are  coarsenings of the $\bZ^2$-grading.

Take a canonical basis of the split Cayley algebra as in Table \ref{ta:splitCayley} and consider the pseudo-octonion algebra $P_8(k)$ in two different ways: as the Petersson algebra $\overline{C(k)}_{\tau_{nst}}$ and as the Petersson algebra $\overline{C(k)}_{\tau_\omega}$ (Propositions \ref{pr:taunst} and \ref{pr:tauomega}). Then the linear map $\phi:\overline{C(k)}_{\tau_\omega}\rightarrow \overline{C(k)}_{\tau_{nst}}$ such that
\[
\begin{aligned}
\phi(e_1)&=e_1,\quad\phi(e_2)=e_2,\quad&\phi(u_3)&=u_3,\quad\phi(v_3)=v_3,\\
\phi(u_1)&=\omega u_1-\omega^2 u_2, & \phi(u_2)&=\frac{1}{\omega-\omega^2}(\omega^2u_1-\omega u_2)\\
 \phi(v_1)&=\frac{1}{\omega^2-\omega}(\omega v_1+\omega^2 v_2), & \phi(v_2)&=\omega^2v_1+\omega v_2,
\end{aligned}
\]
is an isomorphism. Moreover, under this isomorphism, the $5$-grading in Theorem \ref{th:gradingssymmetric} becomes the $5$-grading in $\overline{C(k)}_{\tau_\omega}$ such that $S_0=\espan{e_1,e_2}$, $S_1=\espan{u_1,u_2}$, $S_ 2=\espan{v_3}$, $S_{-1}=\espan{v_1,v_2}$, and $S_{-2}=\espan{u_3}$, which is obviously a coarsening of the $\bZ^2$-grading given in Theorem \ref{th:gradingssymmetric}.

On the other hand, since the ground field $k$ is assumed here to be algebraically closed of characteristic $\ne 3$, the two dimensional subalgebra $K$ that appears in the $\bZ_2^2$-grading is isomorphic to the cartesian product $k\times k$, and the quaternion algebra $Q$ to $\Mat_2(k)$. Hence, a canonical basis can be taken on $C=CD(Q,\alpha)$ such that $K=\espan{e_1,e_2}$, $w=\omega e_1+\omega^2 e_2$, $Q=\espan{e_1,e_2,u_3,v_3}$ and $u=u_1+\alpha v_1$. Then $\tau(u)=wu=(\omega e_1+\omega^2e_2)(u_1+\alpha v_1)=\omega u_1+\omega^2\alpha v_1$, so $\tau(u_1)=\tau(e_1 u)=e_1\tau(u)=\omega u_1$ and $\tau(v_1)=\omega^2v_1$. It follows that $\tau$ is the automorphism $\tau_\omega$ in \eqref{eq:tauomega}, and the Okubo algebra $S=\bar C_\tau$ is just the algebra $S=\overline{C(k)}_{\tau_\omega}$, with the $\bZ_2$-grading given by $S_{(\bar 0,\bar 0)}=K=\espan{e_1,e_2}$, $S_{(\bar 1,\bar 0)}=\espan{u_3,v_3}$, $S_{(\bar 0,\bar 1)}=Ku=\espan{u_1,v_1}$ and $S_{(\bar 1,\bar 1)}=\espan{u_2,v_2}$, and this is clearly a coarsening of the $\bZ^2$-grading. \hfill\qed
\end{remark}

\medskip

\begin{remark}\label{re:Bahturinetal}
The pseudo-octonion algebra was introduced by S.~Okubo in \cite{Okubo78} as follows. Let $k$ be a field of characteristic $\ne 2,3$ containing a root $\mu=\frac{1}{6}(3+\sqrt{-3})$ of the equation $3X(1-X)=1$ (which is equivalent to containing the primitive cubic roots of $1$). On the space $\frsl(3,k)$ of the trace zero $3\times 3$ matrices over $k$ define a new multiplication by:
\[
x*y=\mu xy+(1-\mu)yx-\frac{1}{3}\tr(xy)1,
\]
and a norm $n$ given by $n(x)=\frac{1}{6}\tr(x^2)$. Then $(S,*,n)$ is a composition algebra, which is isomorphic to $P_8(k)$.

It follows that any grading by an abelian group of the algebra of matrices $\Mat_3(k)$ is inherited by $P_8(k)$, since the Cayley-Hamilton equation forces the trace to behave nicely with respect to the grading: given any two homogeneous elements $x,y\in \Mat_3(k)$, $\tr(xy)$ is $0$ unless there is an element $g$ in the grading group such that $x\in\Mat_3(k)_g$ and $y\in\Mat_3(k)_{-g}$. Conversely, any grading of $P_8(k)$ gives a grading of the Lie algebra $P_8(k)^-$ (with bracket $[x,y]^*=x*y-y*x=(2\mu -1)[x,y]$), which is isomorphic to $\frsl_3(k)$ (therefore, the gradings on $\frsl_3(k)$ and on $P_8(k)$ coincide). It also gives a grading on $P_8(k)^+$ (with multiplication $x*y+y*x=xy+yx-\frac{2}{3}\tr(xy)1$), and thus it gives a grading on $\Mat_3(k)$.

Over an algebraically closed field of characteristic $0$, there are just two fine gradings by abelian groups of $\Mat_3(k)$ (see \cite[Theorem 6]{Bahturinetal}): an ``elementary'' $\bZ^2$-grading and a $\bZ_3^2$-grading. These two gradings induce the two fine gradings of $P_8(k)$ considered in Remark \ref{re:coarsenings}. \hfill\qed
\end{remark}

\bigskip

\section{Gradings on exceptional simple Lie algebras}\label{se:GradLie}

The gradings obtained in the previous sections allow us to get some interesting gradings on exceptional simple Lie algebras. First, the construction of the exceptional simple Lie algebras (other than $G_2$) from symmetric composition algebras obtained in \cite{Eld04Ibero} will be reviewed.

The characteristic of the ground field will be considered to be $\ne 2,3$ throughout this section, unless otherwise stated.

\medskip

\subsection{The triality Lie algebra}

Let $(S,*,n)$ be any symmetric composition algebra and consider the corresponding orthogonal Lie algebra:
\[
\fro(S,n)=\{ d\in \End_k(S): n\bigl(d(x),y\bigr)+n\bigl(x,d(y)\bigr)=0\ \forall x,y\in S\},
\]
and the subalgebra of $\fro(S,n)^3$ (with componentwise multiplication) defined by:
\[
\tri(S,*,n)=\{(d_0,d_1,d_2)\in\fro(S,n)^3: d_0(x*y)=d_1(x)*y+x*d_2(y)\ \forall x,y\in S\}.
\]
If the context is clear, we will just write $\tri(S)$.

It turns out (see \cite{Eld04Ibero}) that the map:
\begin{equation}\label{eq:theta}
\begin{split}
\theta:\tri(S,*,n)&\rightarrow \tri(S,*,n)\\
 (d_0,d_1,d_2)&\mapsto (d_2,d_0,d_1)
\end{split}
\end{equation}
is an automorphism of order $3$. The subalgebra of fixed elements is (isomorphic to) the Lie algebra of derivations of $(S,*)$, which is a simple Lie algebra of type $G_2$ if $(S,*,n)$ is a para-Cayley algebra, or of type $A_2$ if $(S,*,n)$ is an Okubo algebra.

Then we have the following result (see \cite[Lemma 2.1]{Eld04Ibero} or \cite[Chapter VIII]{KMRT}):

\begin{theorem}\label{th:triality}
Let $(S,*,n)$ be an eight dimensional symmetric composition algebra over a field of characteristic $\ne 2$. Then:
\begin{romanenumerate}
\item \textbf{(Principle of Local Triality)}\quad The projection $\pi_0:\tri(S,*,n)\rightarrow \fro(S,n):$ $(d_0,d_1,d_2)\mapsto d_0$, is an isomorphism of Lie algebras.

\item For any $x,y\in S$, consider the triple:
\begin{equation}\label{eq:txy}
t_{x,y}=\Big(\sigma_{x,y},\frac{1}{2}n(x,y)id-r_xl_y,\frac{1}{2}q(x,y)id-l_xr_y\Bigr),
\end{equation}
where $\sigma_{x,y}:z\mapsto n(x,z)y-n(y,z)x$. Then
\[
\tri(S,*,n)=t_{S,S}\, (=\espan{t_{x,y}: x,y\in S}),
\]
and
\[
[t_{a,b},t_{x,y}]=t_{\sigma_{a,b}(x),y}+t_{x,\sigma_{a,b}(y)}
\]
for any $a,b,x,y\in S$.\hfill\qed
\end{romanenumerate}
\end{theorem}

\bigskip

Now, given two symmetric composition algebras $(S,*,n)$ and $(S',\star,n')$, one can form the Lie algebra
\begin{equation}\label{eq:gSS'}
\frg=\frg(S,S')=\bigl(\tri(S)\oplus\tri(S')\bigr)\oplus\Bigl(\oplus_{i=0}^2\iota_i(S\otimes S')\Bigr),
\end{equation}
where $\iota_i(S\otimes S')$ is just a copy of $S\otimes S'$ ($i=0,1,2$) (and unadorned tensor products are always considered over the ground field $k$), with bracket given by:

\smallskip
\begin{itemize}
\item the Lie bracket in $\tri(S)\oplus\tri(S')$, which thus becomes  a Lie subalgebra of $\frg$,
\smallskip

\item $[(d_0,d_1,d_2),\iota_i(x\otimes
 x')]=\iota_i\bigl(d_i(x)\otimes x'\bigr)$,
\smallskip

\item
 $[(d_0',d_1',d_2'),\iota_i(x\otimes
 x')]=\iota_i\bigl(x\otimes d_i'(x')\bigr)$,
\smallskip

\item $[\iota_i(x\otimes x'),\iota_{i+1}(y\otimes y')]=
 \iota_{i+2}\bigl((x* y)\otimes (x'\star y')\bigr)$ (indices modulo
 $3$),
\smallskip

\item $[\iota_i(x\otimes x'),\iota_i(y\otimes y')]=
 n'(x',y')\theta^i(t_{x,y})+
 n(x,y)\theta'^i(t'_{x',y'})$,

\end{itemize}
\smallskip

\noindent
for any $i=0,1,2$ and elements $x,y\in S$, $x',y'\in S'$,
$(d_0,d_1,d_2)\in\tri(S)$, and $(d_0',d_1',d_2')\in\tri(S')$. Here
$\theta$ denotes the automorphism considered in \eqref{eq:theta}
in $\tri(S)$, and $\theta'$ the analogous automorphism in $\tri(S')$, while $t_{x,y}$ is defined by \eqref{eq:txy},
and $t'_{x',y'}$ denotes the analogous elements in $\tri(S')$.

The type of the Lie algebra thus obtained is given in Table \ref{ta:FMS}, which is the classical Freudenthal Magic Square, and where $S_i$ or $S_i'$ denotes a symmetric composition algebra of dimension $i$.

\begin{table}[h!]
$$
\vbox{\offinterlineskip
 \halign{\hfil\ $#$\quad \hfil&%
 \vreglon #%
 &\hfil\quad $#$\quad \hfil&\hfil$#$\hfil
 &\hfil\quad $#$\quad \hfil&\hfil\quad $#$\quad \hfil\cr
 \bigstrut &width 0pt&S_1&S_2&S_4&S_8\cr
 &\multispan5{\hreglonfill}\cr
 S_1'&&A_1&A_2&C_3&F_4\cr
 \bigstrut S_2'&&A_2 &A_2\oplus A_2&A_5&E_6\cr
 \bigstrut S_4'&&C_3&A_5 &D_6&E_7\cr
 \bigstrut S_8'&&F_4&E_6& E_7&E_8\cr}}
$$
\bigskip
\caption{Freudenthal Magic Square}\label{ta:FMS}
\end{table}

The Lie algebra $\frg(S,S')$ is naturally $\bZ_2\times\bZ_2$-graded with
\begin{equation}\label{eq:Z2Z2gSS'}
\begin{aligned}
&&\frg_{(\bar 0,\bar 0)}&=\tri(S)\oplus\tri(S'),&&\\
\frg_{(\bar 1,\bar 0)}&=\iota_0(S\otimes S'),\quad&
\frg_{(\bar 0,\bar 1)}&=\iota_1(S\otimes S'),&
\frg_{(\bar 1,\bar 1)}&=\iota_2(S\otimes S').
\end{aligned}
\end{equation}

Also, the order $3$ automorphisms $\theta$ and $\theta'$ extend to an order $3$ automorphism $\Theta$ of $\frg(S,S')$ such that its restriction to $\tri(S)$ (respectively $\tri(S')$) is $\theta$ (respectively $\theta'$), while
\begin{equation}\label{eq:Theta}
\Theta\bigl(\iota_i(x\otimes x')\bigr)=\iota_{i+1}(x\otimes x')
\end{equation}
for any $x\in S$, $x'\in S'$ and $i=0,1,2$ (indices modulo $3$).

If the ground field contains the cubic roots of $1$, the  eigenspaces of $\Theta$ constitute a $\bZ_3$-grading of $\frg(S,S')$.

\begin{remark}\label{re:gSStauS'}
Let $\tau$ be an automorphism of a Hurwitz algebra $(C,\cdot,n)$ with $\tau^3=1$. Consider the para-Hurwitz algebra $S^{\bullet}=\bar C$, with multiplication $x\bullet y=\bar x\cdot\bar y$, and the symmetric composition algebra $S^*$ with multiplication $x*y=\tau(\bar x)\cdot \tau^2(\bar y)$ as in \eqref{eq:Petersson}. Then it is shown in \cite{EldNewLook} that given any other symmetric composition algebra $(S',\star,n')$, the Lie algebras $\frg(S^\bullet,S')$ and $\frg(S^*,S')$ are isomorphic as $\bZ_2\times\bZ_2$-graded Lie algebras.\hfill\qed
\end{remark}

\medskip

\subsection{Gradings on $D_4$}

Assume here that the ground field, besides being of characteristic $\ne 2,3$, contains the cubic roots of $1$.

\smallskip

Let  $(C,\cdot,n)$ be a Cayley algebra with the $\bZ_2^3$-grading in Theorem \ref{th:gradingsCayley} and consider the corresponding para-Cayley algebra $\bar C$, which inherits this grading. This, in turn, induces a grading on the orthogonal Lie algebra $\fro(C,n)$, with
\begin{equation}\label{eq:oCnZ23}
\fro(C,n)_\mu=\{d\in\fro(C,n): d(C_{\gamma})\subseteq C_{\gamma+\mu}\ \forall\gamma\in\bZ_2^3\}.
\end{equation}
It is straightforward to check that $\fro(C,n)_{(\bar 0,\bar 0,\bar 0)}=0$, while $\fro(C,n)_\mu$ is a Cartan subalgebra of $\fro(C,n)$ for any $0\ne \mu\in \bZ_2^3$. Thus, this $\bZ_2^3$-grading is a grading where all its nonzero homogeneous components are Cartan subalgebras. Borrowing the definition in \cite{Thompson}, this will be called a \emph{Dempwolff decomposition}. Such a grading is called \emph{very pure} in \cite{HesselinkPureSpecial}.

Since we have $7$ homogeneous components, all of them of dimension $4$, the \emph{type} of this grading is $(0,0,0,7)$. The type $(h_1,\ldots,h_l)$ of a grading indicates that there are $h_i$ homogeneous components of dimension $i$, for any $i=1,\ldots,l$ (see \cite{HesselinkPureSpecial}).

\smallskip

The order $3$ automorphism $\theta$ of $\tri(\bar C)$ in \eqref{eq:theta} induces a $\bZ_3$-grading of $\tri(\bar C)$ whose homogeneous components are the eigenspaces of $\theta$. This is compatible with the $\bZ_2^3$-grading above, thus appearing a $\bZ_2^3\times\bZ_3$-grading of $\frt=\tri(\bar C)$ where, if $\omega^3=1\ne \omega$:
\[
\begin{split}
\frt_{(\mu,\bar \jmath)}=\{(d_0,d_1,d_2)\in \frt :  d_i&\in\fro(C,n)_\mu\ (i=0,1,2)\ \text{and}\\
& \theta\bigl((d_0,d_1,d_2)\bigr)=\omega^j(d_0,d_1,d_2)\}
\end{split}
\]
for any $\mu\in\bZ_2^3$ and $j=0,1,2$.

Since the projection $\pi_0:\tri(\bar C)\rightarrow \fro(C,n):$ $(d_0,d_1,d_2)\mapsto d_0$, is an isomorphism (Theorem \ref{th:triality}), we obtain in this way a $\bZ_2^3\times \bZ_3$-grading on $\fro(C,n)$. Denote again by $\theta$ the order $3$ automorphism induced by $\theta$ on $\fro(C,n)$.

As shown in \cite[(3.79)]{Schafer}, $\fro(C,n)=\der C\oplus L_{C_0}\oplus R_{C_0}$, where $C_0=\{ x\in C: n(x,1)=0\}$ and $L_x$ and $R_x$ denote the left and right multiplications by $x$ in $C$. Thus, as a $\der C$-module, $\fro(C,n)$ is the direct sum of the adjoint module and of two copies of the seven dimensional irreducible module $C_0$.

Our assumption on the characteristic of the ground field  being different from $2$ and $3$ implies that the Killing form $\kappa$ of $\fro(C,n)$ is nondegenerate. Besides, $\kappa\bigl(\fro(C,n)_\mu,\fro(C,n)_\nu\bigr)=0$ unless $\mu+\nu=0$, $\mu,\nu\in\bZ_2^3\times \bZ_3$.

Moreover, $\der C=\{d\in\fro(C,n): \theta(d)=d\}$ holds, and by complete reducibility the eigenspaces $\fro(C,n)_{\omega^j}=\{d\in\fro(C,n): \theta(d)=\omega^jd\}$ ($j=1,2$) are isomorphic, as modules for the Lie algebra $\der C$, to the seven dimensional irreducible module $C_0$. It follows that $\dim\fro(C,n)_{(\mu,\bar \jmath)}=1$ for any $0\ne\mu\in\bZ_2^3$ and $j=1,2$ and $\fro(C,n)_{(0,\bar \jmath)}=0$ for any $j$, and hence $\dim\fro(C,n)_{(\mu,\bar 0)}=\dim(\der C)_\mu=2$ for any $0\ne\mu\in\bZ_2^3$ (so this is a Dempwolff decomposition, or a very pure grading, of $\der C$, a simple Lie algebra of type $G_2$). Therefore, the next result follows:

\begin{proposition}\label{pr:Z23D4}
A $\bZ_2^3$-grading of a para-Cayley algebra $(\bar C,\bullet,n)$ over a field $k$ of characteristic $\ne 2,3$ containing the cubic roots of $1$ induces a $\bZ_2^3\times\bZ_3$-grading of the orthogonal Lie algebra $\fro(C,n)$ of type $(14,7)$.
\end{proposition}

\smallskip

For an algebraically closed field of characteristic $0$, the gradings of the orthogonal Lie algebra $\fro(8)$ have been recently determined in \cite{DMV}. The above gives a nice description of the unique grading of type $(14,7)$ in \cite{DMV}.

\bigskip

Consider now an Okubo algebra with isotropic norm $(\calO,*,n)$, endowed with a standard $\bZ_3^2$-grading (as in \eqref{eq:Z32standard}). In this case the orthogonal Lie algebra $\fro(\calO,n)$ is $\bZ_3^2$-graded too with
\[
\fro(\calO,n)_\mu=\{d\in\fro(\calO,n): d(\calO_\gamma)\subseteq \calO_{\gamma+\mu}\ \forall\gamma\in\bZ_3^2\},
\]
and $\fro(\calO,n)_0$ is a Cartan subalgebra of $\fro(\calO,n)$ (in particular $\dim\fro(\calO,n)_0=4$), while $\dim\fro(\calO,n)_\mu=3$ for any $0\ne \mu\in\bZ_3^2$. That is, the type is $(0,0,8,1)$.

Again, the automorphism $\theta$ refines this grading to a $\bZ_3^3$-grading, where
\[
\fro(\calO,n)_{(\bar \jmath_1,\bar \jmath_2,\bar \jmath_3)}=
 \{d\in \fro(\calO,n)_{(\bar \jmath_1,\bar \jmath_2)}: \theta(d)=\omega^{j_3}d\}.
\]
Here the $1$-eigenspace $\{d\in\fro(\calO,n): \theta(d)=d\}=\der(\calO,*)$ is the Lie algebra $\ad_{\calO}^*=\{\ad_x^*:x\in \calO\}$ (see \cite{EPLargeDer}), where $\ad_x^*(y)=x*y-y*x=(l_x-r_x)(y)$, so that here the $\bZ_3^3$-grading is of type $(8)$. Using that the $\omega$ and $\omega^2$-eigenspaces relative to $\theta$ are dual relative to the Killing form of $\fro(\calO,n)$, it follows that $\fro(\calO,n)_{(\bar 0,\bar 0,\bar 0)}=0$, that $\dim\fro(\calO,n)_{(\bar 0,\bar 0,\bar \jmath)}=2$ for $j=1,2$, and that $\dim\fro(\calO,n)_{(\bar \jmath_1,\bar \jmath_2,\bar \jmath_3)}=1$ if $(\bar \jmath_1,\bar \jmath_2)\ne (\bar 0,\bar 0)$. Thus:

\begin{proposition}\label{pr:Z33o8}
The standard $\bZ_3^2$-grading on an Okubo algebra $(\calO,*,n)$ over a field of characteristic $\ne 2,3$ containing the cubic roots of $1$ induces a $\bZ_3^3$-grading on the orthogonal Lie algebra $\fro(\calO,n)$ of type $(24,2)$.
\end{proposition}

\smallskip

For an algebraically closed field of characteristic $0$, this gives a concrete description of the unique $\bZ_3^3$-grading of $\fro(8)$ in \cite{DMV}.

\medskip

\subsection{Gradings on $F_4$}

Let $(C,\cdot,n)$ be a Cayley algebra and consider the Albert algebra (or exceptional Jordan algebra) $J=H_3(C)$ of $3\times 3$ hermitian matrices over $C$, relative to the involution $(a_{ij})^*=(\bar a_{ij})$. This is a Jordan algebra under the multiplication
\begin{equation}\label{eq:Jproduct}
x\circ y=\frac{1}{2}(xy+yx).
\end{equation}
Consider the corresponding para-Cayley algebra $(\bar C,\bullet,n)$ ($a\bullet b=\bar a\cdot\bar b$).
Then,
\begin{equation}\label{eq:JH3C}
\begin{split}
J=H_3(C)&=\left\{ \begin{pmatrix} \alpha_0 &\bar a_2& a_1\\
  a_2&\alpha_1&\bar a_0\\ \bar a_1&a_0&\alpha_2\end{pmatrix} :
  \alpha_0,\alpha_1,\alpha_2\in k,\ a_0,a_1,a_2\in C\right\}\\[6pt]
 &=\bigl(\oplus_{i=0}^2 ke_i\bigr)\oplus
     \bigl(\oplus_{i=0}^2\iota_i(S)\bigr),
\end{split}
\end{equation}
where
\begin{equation}\label{eq:eisiotas}
\begin{aligned}
e_0&= \begin{pmatrix} 1&0&0\\ 0&0&0\\ 0&0&0\end{pmatrix}, &
 e_1&=\begin{pmatrix} 0&0&0\\ 0&1&0\\ 0&0&0\end{pmatrix}, &
 e_2&= \begin{pmatrix} 0&0&0\\ 0&0&0\\ 0&0&1\end{pmatrix}, \\
 \iota_0(a)&=2\begin{pmatrix} 0&0&0\\ 0&0&\bar a\\
 0&a&0\end{pmatrix},&
 \iota_1(a)&=2\begin{pmatrix} 0&0&a\\ 0&0&0\\
 \bar a&0&0\end{pmatrix},&
 \iota_2(a)&=2\begin{pmatrix} 0&\bar a&0\\ a&0&0\\
 0&0&0\end{pmatrix},
\end{aligned}
\end{equation}
for any $a\in C$. Identify $ke_0\oplus ke_1\oplus ke_2$ to $k^3$ by
means of $\alpha_0e_0+\alpha_1e_1+\alpha_2e_2\simeq
(\alpha_0,\alpha_1,\alpha_2)$. Then the commutative
multiplication \eqref{eq:Jproduct} becomes:
\begin{equation}\label{eq:Jniceproduct}
\left\{\begin{aligned}
 &(\alpha_0,\alpha_1,\alpha_2)\circ(\beta_1,\beta_2,\beta_3)=
    (\alpha_0\beta_0,\alpha_1\beta_1,\alpha_2\beta_2),\\
 &(\alpha_0,\alpha_1,\alpha_2)\circ \iota_i(a)
  =\frac{1}{2}(\alpha_{i+1}+\alpha_{i+2})\iota_i(a),\\
 &\iota_i(a)\circ\iota_{i+1}(b)=\iota_{i+2}(a\bullet b),\\
 &\iota_i(a)\circ\iota_i(b)=2n(a,b)\bigl(e_{i+1}+e_{i+2}\bigr),
\end{aligned}\right.
\end{equation}
for any $\alpha_i,\beta_i\in k$, $a,b\in C$, $i=0,1,2$, and where
indices are taken modulo $3$.

This shows a natural $\bZ_2^2$-grading on $J$ with:
\[
J_{(\bar 0,\bar 0)}=k^3,\quad J_{(\bar 1,\bar 0)}=\iota_0(C),\quad
J_{(\bar 0,\bar 1)}=\iota_1(C),\quad J_{(\bar 1,\bar 1)}=\iota_2(C).
\]

Actually, a more general situation can be considered, which has its own independent interest:

\begin{theorem}\label{th:Jsymmetriccompo}
Let $(S,*,n)$ be any symmetric composition algebra over a field $k$ of characteristic $\ne 2$. On the vector space $\bA=\bA(S)=k^3\oplus\bigl(\oplus_{i=0}^2\iota_i(S)\bigr)$ ($\iota_i(S)$ is just a copy of $S$) define a commutative multiplication by the formulas:
\begin{equation}\label{eq:Jniceproductbis}
\left\{\begin{aligned}
 &(\alpha_0,\alpha_1,\alpha_2)\circ(\beta_1,\beta_2,\beta_3)=
    (\alpha_0\beta_0,\alpha_1\beta_1,\alpha_2\beta_2),\\
 &(\alpha_0,\alpha_1,\alpha_2)\circ \iota_i(a)
  =\frac{1}{2}(\alpha_{i+1}+\alpha_{i+2})\iota_i(a),\\
 &\iota_i(a)\circ\iota_{i+1}(b)=\iota_{i+2}(a* b),\\
 &\iota_i(a)\circ\iota_i(b)=2n(a,b)\bigl(e_{i+1}+e_{i+2}\bigr),
\end{aligned}\right.
\end{equation}
for any $\alpha_i,\beta_i\in k$, $a,b\in S$, $i=0,1,2$, and where
indices are taken modulo $3$.

Then $\bA$ is a central simple Jordan algebra.
\end{theorem}
\begin{proof}
By extending scalars, we may assume that the ground field $k$ is algebraically closed and it is enough to deal with the eight dimensional case (as any lower dimensional symmetric composition algebra is a subalgebra of an eight dimensional one). In the para-Cayley case, $\bA$ is isomorphic to the algebra $J$ above. Otherwise $S$ is the pseudo-octonion algebra $P_8(k)=\overline{C(k)}_{\tau_{st}}$ with multiplication $a*b=\tau_{st}(\bar a)\tau_{st}^2(\bar b)=\tau_{st}(a)\bullet\tau_{st}^2(b)$.

Consider the algebra $J=H_3(C(k))$ as above. Then the linear map:
\[
\begin{split}
\Phi:J&\rightarrow \bA,\\
(\alpha_0,\alpha_1,\alpha_2)&\mapsto  (\alpha_0,\alpha_1,\alpha_2),\\
\iota_i(a)&\mapsto \iota_i(\tau_{st}^{-i}(a)),
\end{split}
\]
is easily seen to be an isomorphism. For instance,
\[
\Phi(\iota_i(a)\circ\iota_{i+1}(b))=\Phi(\iota_{i+2}(a\bullet b))=\iota_{i+2}(\tau_{st}^{-i-2}(a\bullet b)),
\]
while
\[
\begin{split}
\Phi(\iota_i(a))\circ\Phi(\iota_{i+1}(b))
 &=\iota_i(\tau_{st}^{-i}(a))\circ\iota_{i+1}(\tau_{st}^{-i-1}(b))\\
 &=\iota_{i+2}\bigl(\tau_{st}^{-i}(a)*\tau_{st}^{-i-1}(b)\bigr)\\
 &=\iota_{i+2}\bigl(\tau_{st}(\tau_{st}^{-i}(a))\bullet\tau_{st}^2(\tau_{st}^{-i-1}(b))\bigr)\\
 &=\iota_{i+2}(\tau_{st}^{-i-2}(a\bullet b)),
\end{split}
\]
for any $i=0,1,2$ (indices modulo $3$), and $a,b\in C(k)$.
\end{proof}

\smallskip

Note that the Jordan algebra $\bA(S)$ above is naturally endowed with an order $3$ automorphism $\theta$ such that $\theta\bigl((\alpha_0,\alpha_1,\alpha_2)\bigr)=(\alpha_2,\alpha_0,\alpha_1)$, and $\theta\bigl(\iota_i(a)\bigr)=\iota_{i+1}(a)$ for any $a\in S$ and $i=0,1,2$.

\medskip

Now, given an eight dimensional symmetric composition algebra $(S,*,n)$, consider the Albert algebra $J=\bA(S)$. Its Lie algebra of derivations, which is a simple Lie algebra of type $F_4$, is isomorphic to the Lie algebra $\frg(k,S)$ as follows (see \cite[Section 3]{CE2} or \cite[Section 2]{IvanProc}): First, the $\bZ_2^2$-grading on $J=\bA(S)$ induces a $\bZ_2^2$-grading on $\der J$:
\[
(\der J)_\mu=\{d\in\der J: d(J_{\gamma})\subseteq J_{\gamma+\mu}\ \forall\gamma\in \bZ_2^2\},
\]
and the linear map:
\[
\begin{split}
\phi:\tri(S)&\longrightarrow (\der J)_{(\bar 0,\bar 0)}\\
(d_0,d_1,d_2)&\mapsto D_{(d_0,d_1,d_2)},
\end{split}
\]
where
\begin{equation}\label{eq:Dd0d1d2}
\left\{\begin{aligned} &D_{(d_0,d_1,d_2)}(e_i)=0,\\
   &D_{(d_0,d_1,d_2)}\bigl(\iota_i(a)\bigr)=
   \iota_i\bigl(d_i(a)\bigr)
   \end{aligned}\right.
\end{equation}
for any $i=0,1,2$ and $a\in S$, is an isomorphism.

Now, for any $a\in S$ and $i=0,1,2$, consider the derivation:
\[
D_i(a)=2[L_{\iota_i(a)},L_{e_{i+1}}]
\]
(indices modulo $3$), where $L_x$ denotes the multiplication by $x$
in $J$. Then (see \cite[(2.10)]{IvanProc}):
\begin{equation}\label{eq:Diaaction}
\begin{split}
&D_i(a)(e_i)=0,\ D_i(a)(e_{i+1})=\frac{1}{2} \iota_i(a),\
  D_i(a)(e_{i+2})=-\frac{1}{2}\iota_i(a),\\
&D_i(a)\bigl(\iota_{i+1}(b)\bigr)=-\iota_{i+2}(a* b),\\
&D_i(a)\bigl(\iota_{i+2}(b)\bigr)=\iota_{i+1}(b* a),\\
&D_i(a)\bigl(\iota_i(b)\bigr)=2n(a,b)(-e_{i+1}+e_{i+2}),
\end{split}
\end{equation}
for any $i=0,1,2$ and any $a,b\in S$.

The isomorphism $\phi$ above extends to a Lie algebra isomorphism (see \cite[Theorem 3.13]{CE2}):
\begin{equation}\label{eq:PhigkSderJ}
\begin{split}
\Phi:\frg(k,S)=\tri(S)\oplus\bigl(\oplus_{i=0}^2\iota_i(k\otimes S)\bigr)&\longrightarrow \der J\\
\tri(S)\ni (d_0,d_1,d_2)&\mapsto D_{(d_0,d_1,d_2)}\\
\iota_i(k\otimes S)\ni\iota_i(1\otimes a)&\mapsto D_i(a).
\end{split}
\end{equation}

\bigskip

We can combine now the $\bZ_2^3$-grading on a para-Cayley algebra $\bar C$ with the $\bZ_2^2$-grading on either the Albert algebra $J=\bA(\bar C)$ or its Lie algebra of derivations $\der J$ to obtain a $\bZ_2^5$-grading on each of these latter algebras.

For the Albert algebra, given any $\mu\in\bZ_2^3$:
\[
\begin{split}
\bA(\bar C)_{(0,\bar 0,\bar 0)}&=k^3,\\
\bA(\bar C)_{(\mu,\bar 1,\bar 0)}&=\iota_0(\bar C_{\mu}),\\
\bA(\bar C)_{(\mu,\bar 0,\bar 1)}&=\iota_1(\bar C_{\mu}),\\
\bA(\bar C)_{(\mu,\bar 1,\bar 1)}&=\iota_2(\bar C_{\mu}),
\end{split}
\]
so this $\bZ_2^5$-grading is a grading of type $(24,0,1)$. For its Lie algebra of derivations, the grading induced in $\tri(\bar C)$ corresponds, by means of the isomorphism $\pi_0:\tri(\bar C)\rightarrow \fro(C,n)$, to the $\bZ_2^3$-grading in \eqref{eq:oCnZ23}, while $(\der J)_{(\mu,\bar 1,\bar 0)}=D_0(\bar C_{\mu})$, $(\der J)_{(\mu,\bar 0,\bar 1)}=D_1(\bar C_{\mu})$ and $(\der J)_{(\mu,\bar 1,\bar 1)}=D_2(\bar C_{\mu})$. (Alternatively, through the isomorphism $\Phi$ in \eqref{eq:PhigkSderJ}, $\frg(k,\bar C)_{(\mu,\bar 1,\bar 0)}=\iota_0(k\otimes\bar C_\mu)$, ...)

Thus, we obtain a grading of type $(24,0,0,7)$ on the central simple Lie algebra $\der J$.

\bigskip

On the other hand, assuming the ground field contains the cubic roots of $1$, we can combine the standard $\bZ_3^2$-grading on an Okubo algebra $(\calO,*,n)$ with the $\bZ_3$-grading given by the natural order $3$ automorphism $\Theta$ in $\frg(k,S)\simeq \der\bA(\calO)$ in \eqref{eq:Theta}, or the natural order $3$ automorphism in $\bA(\calO)$, to get a $\bZ_3^3$-grading of $\der\bA(\calO)$ or $\bA(\calO)$. For the Albert algebra $\bA(\calO)$, given any $0\ne\mu\in\bZ_3^2$:
\[
\begin{split}
\bA(\calO)_{(\bar 0,\bar 0,\bar 0)}&=k(e_1+e_1+e_2)=k1,\\
\bA(\calO)_{(\bar 0,\bar 0,\bar 1)}&=k(e_0+\omega^2e_1+\omega e_2),\\
\bA(\calO)_{(\bar 0,\bar 0,\bar 2)}&=k(e_0+\omega e_1+\omega^2e_2),\\
\bA(\calO)_{(\mu,\bar 0)}&=k\bigl(\sum_{i=0}^2\iota_i(a_\mu)\bigr),\\
\bA(\calO)_{(\mu,\bar 1)}
 &=k\bigl(\iota_0(a_\mu)+\omega^2\iota_1(a_\mu)+\omega\iota_2(a_\mu)\bigr),\\
\bA(\calO)_{(\mu,\bar 2)}
 &=k\bigl(\iota_0(a_\mu)+\omega\iota_1(a_\mu)+\omega^2\iota_2(a_\mu)\bigr),
\end{split}
\]
where $\omega^3=1\ne \omega$ and $a_\mu$ denotes a nonzero element in the one dimensional homogeneous space $\calO_\mu$. So the type of this $\bZ_3^3$-grading is $(27)$.

For the Lie algebra of derivations, the grading induced in $\tri(\calO)$ is the one in Proposition \ref{pr:Z33o8} of type $(24,2)$, while the subspace $\oplus_{i=0}^2D_i(\calO)$ decomposes into the direct sum of another $24$ homogeneous spaces of dimension $1$ with degrees $(\mu,\bar \jmath)$, $0\ne\mu\in\bZ_3^2$, $j=0,1,2$ (same degrees as the $24$ homogeneous one dimensional spaces in $\tri(\calO)$). It follows that the type of the $\bZ_3^3$-grading on $\der\bA(\calO)$ is $(0,26)$.

Summarizing, we have obtained the following result:

\begin{proposition}\label{pr:gradingsF4}
Let $k$ be a field of characteristic $\ne 2,3$ containing the cubic roots of $1$:
\begin{enumerate}
\item A $\bZ_2^3$-grading in a para-Cayley algebra $\bar C$ induces $\bZ_2^5$-gradings on the Albert algebra $\bA(\bar C)$ and on its Lie algebra of derivations of respective types $(24,0,1)$ and $(24,0,0,7)$.

\item A standard $\bZ_3^2$-grading on an Okubo algebra $\calO$ induces $\bZ_3^3$-gradings on the Albert algebra $\bA(\calO)$ and on its Lie algebra of derivations of respective types $(27)$ and $(0,26)$.
\end{enumerate}
\end{proposition}

\medskip

\begin{remark}
\null\quad
\begin{enumerate}
\item For an algebraically closed field of characteristic $0$, the gradings obtained in the previous Proposition are among the four fine gradings on either the Albert algebra or the exceptional simple Lie algebra of type $F_4$ considered in \cite{DraperMartinF4}.

\item A $\bZ_2^3$-grading on a para-Cayley algebra $\bar C$ over a field of characteristic $\ne 2,3$ containing the cubic roots of $1$ also induces $\bZ_2^3\times \bZ_3$-gradings on $\bA(\bar C)$ and on $\frg(k,\bar C)\simeq \der \bA(\bar C)$ of types $(21,3)$ and $(3,14,7)$. However these gradings are not fine. Similarly, a standard $\bZ_3^2$-grading on an Okubo algebra $\calO$ induces $\bZ_3^2\times \bZ_2^2$-gradings on $\bA(\calO)$ and $\frg(k,\calO)\simeq\der\bA(\calO)$ of types $(24,0,1)$ and $(24,0,8,1)$. The unique four dimensional homogeneous space in $\frg(k,\calO)$ is the one corresponding to the neutral element $\frg(k,\calO)_0$, which is a Cartan subalgebra (inside $\tri(\calO)\simeq \fro(\calO,n)$. It turns out then that, over an algebraically closed field, this grading can be refined to the $\bZ^4$-grading given by the roots relative to this Cartan subalgebra and, in the same vein, the grading on $\bA(\calO)$ can be refined to the $\bZ^4$-grading given by the weights of $\bA(\calO)$ relative to the action of $\frg(k,\calO)\simeq\der\bA(\calO)$.

\item For an algebraically closed field $k$ of characteristic $0$, Draper and Mart\'{\i}n \cite{DraperMartinF4} have shown that there are exactly four fine gradings of the exceptional simple Lie algebra of type $F_4$. These are the Cartan grading ($\bZ^4$-grading given by the roots relative to a Cartan subalgebra), the $\bZ_2^5$ and $\bZ_3^3$-gradings in Proposition \ref{pr:gradingsF4}, and a further $\bZ_2^3\times\bZ$-grading. A concrete description of this latter grading can be obtained using the ingredients here: Given a para-Cayley algebra $\bar C$ with a $\bZ_2^3$-grading, consider again the Jordan algebra $J=\bA(\bar C)$ and its Lie algebra of derivations $\der J\simeq \tri(\bar C)\oplus\bigl(\oplus_{i=0}^2D_i(\bar C)\bigr)$.
    The derivation $D_0(1)$ acts on $J$ as follows:
    $D_0(1)(e_0)=0$, $D_0(1)(e_1)=\frac{1}{2}\iota_0(1)$,
    $D_0(1)(e_2)=-\frac{1}{2}\iota_0(1)$,
    and $D_0(1)(\iota_0(1))=4(-e_1+e_2)$, $D_0(1)(\iota_0(x))=0$ if $\bar x=-x$, $D_0(1)(\iota_1(x))=-\iota_2(\bar x)$, and $D_0(1)(\iota_2(x))=\iota_1(\bar x)$. That is,
    \[
    \begin{split}
    &D_0(1)(e_0)=0=D_0(1)(e_1+e_2),\\
    &D_0(1)(e_1-e_2)=\iota_0(1),\ D_0(1)(\iota_0(1))=-4(e_1-e_2),\\
    &D_0(1)(\iota_0(x)=0\ \text{if $\bar x=-x$,}\\
    &D_0(1)(\iota_1(x))=-\iota_2(\bar x),\ D_0(1)(\iota_2(x))=\iota_1(\bar x),
    \end{split}
    \]
    so that, assuming $\sqrt{-1}\in k$, $D_0(1)$ acts with eigenvalues $0,\pm\sqrt{-1},\pm2\sqrt{-1}$, thus inducing a $\bZ$-grading on $\bA(\bar C)$, and hence another one on $\der\bA(\bar C)$ too, which is compatible with the $\bZ_2^3$ grading induced by the grading on $\bar C$. Thus they combine to give  $\bZ_2^3\times \bZ$-gradings on $\bA(\bar C)$ and $\der\bA(\bar C)$ of types $(25,1)$ and $(31,0,7)$.
\end{enumerate}

\end{remark}

\medskip

Let us have a closer look at the grading of type $(0,26)$ in the Lie algebra $\der\bA(\calO)\simeq \frg(k,\calO)$ induced by the standard $\bZ_3^2$-grading in $\calO$ and the order $3$ automorphism $\Theta$ in \ref{eq:Theta}.

Consider two elements $0\ne \mu,\nu\in \bZ_3^2$ with $\nu\ne\pm\mu$, and let $0\ne x\in \calO_\mu$, $0\ne y\in \calO_\nu$. Then $x$ and $y$ are in the situation of Theorem \ref{th:bonito}. In particular, either $x*y=0$ or $y*x=0$. Assume, without loss of generality, that $x*y=0$. Let $\frt=\tri(\calO,*,n)$ and consider the $\bZ_3^3$-grading induced on $\frt\simeq \fro(\calO,n)$ of type $(24,2)$ in Proposition \ref{pr:Z33o8}. The one-dimensional homogeneous space $\frt_{(\mu+\nu,\bar \jmath)}$ is spanned by the element
\[
(1+\omega^{2j}\theta+\omega^j\theta^2)(t_{x,y})=(d_0,d_1,d_2),
\]
where
\begin{equation}\label{eq:tmunuj}
\begin{split}
d_0&=\sigma_{x,y}
  +\omega^j(\frac{1}{2}n(x,y)1-r_xl_y)+\omega^{2j}(\frac{1}{2}n(x,y)1-l_xr_y),\\
d_1&=\omega^{2j}\sigma_{x,y}
  +(\frac{1}{2}n(x,y)1-r_xl_y)+\omega^j(\frac{1}{2}n(x,y)1-l_xr_y),\\
d_2&=\omega^j\sigma{x,y}
  +\omega^{2j}(\frac{1}{2}n(x,y)1-r_xl_y)+(\frac{1}{2}n(x,y)1-l_xr_y).
\end{split}
\end{equation}

\begin{lemma}\label{le:toral}
With $x,y$ as above, the endomorphism of $\calO$ given by
\[
\epsilon\sigma_{x,y}+\delta(\frac{1}{2}n(x,y)1-r_xl_y)+\gamma(\frac{1}{2}n(x,y)1-l_xr_y)
\]
is semisimple for any $0\ne \epsilon,\delta,\gamma\in k$. Moreover, the kernel of this endomorphism is the subalgebra generated by $y*x$.

In particular, the endomorphisms $d_0,d_1,d_2$ in \eqref{eq:tmunuj} are semisimple.
\end{lemma}

\begin{proof}
Let $\alpha=n(x,x*x)$ and $\beta=n(y*y)$, and consider the basis $\{x,x*x,y,y*y,y*x,(y*y)*(x*x),x*(y*y),(x*x)*^y\}$ as in \eqref{eq:basisxy}. The multiplication table in this basis is given by Table \ref{ta:Oalphabeta} with $x=-x_{1,0}$ and $y=-x_{0,1}$. Then the coordinate matrix of $\epsilon\sigma_{x,y}+\delta(\frac{1}{2}n(x,y)1-r_xl_y)+\gamma(\frac{1}{2}n(x,y)1-l_xr_y)$ is
\[
\left( \begin{tabular}{cc|cc|cc|cc}
 $0$&$0$&$0$&$-\epsilon\beta$&$0$&$0$&$0$&$0$\\
 $0$&$0$&$0$&$0$&$0$&$0$&$-\delta\beta$&$0$\\ \hline
 $0$&$\epsilon\alpha$&$0$&$0$&$0$&$0$&$0$&$0$\\
 $0$&$0$&$0$&$0$&$0$&$0$&$0$&$\gamma\alpha$\\ \hline
 $0$&$0$&$0$&$0$&$0$&$0$&$0$&$0$\\
 $0$&$0$&$0$&$0$&$0$&$0$&$0$&$0$\\ \hline
 $0$&$0$&$-\gamma$&$0$&$0$&$0$&$0$&$0$\\
 $\delta$&$0$&$0$&$0$&$0$&$0$&$0$&$0$\\
 \end{tabular} \right)
\]
and it follows that (renaming the basic elements as $e_1,\ldots,e_8$) the given endomorphism acts as follows:
\[
\begin{split}
&e_1\mapsto \delta e_3\mapsto \delta\gamma\alpha e_4\mapsto
   -\epsilon\delta\gamma\alpha\beta e_1,\\
&e_2\mapsto \epsilon\alpha e_3\mapsto -\epsilon\gamma\alpha e_7\mapsto
    \epsilon\delta\gamma\alpha\beta e_2,\\
&e_5,e_6\mapsto 0.
\end{split}
\]
Thus the minimal polynomial is $(X^3+\epsilon\delta\gamma\alpha\beta)(X^3-\epsilon\delta\gamma\alpha\beta)$ and, since the characteristic is assumed to be $\ne 3$, this polynomial is separable. The result follows now easily.
\end{proof}

\medskip

\begin{corollary}\label{co:tmutoral}
Under the conditions above, the homogenous subspaces $\frt_\mu$ in the Lie algebra $\frg(k,\calO)\simeq \der\bA(\calO)$, $\mu=(\bar\jmath_1,\bar\jmath_2,\bar\jmath_3)\in \bZ_3^3$, act semisimply on the subspace $\oplus_{i=0}^2\iota_i(k\otimes\calO)$.

For $(\bar\jmath_1,\bar\jmath_2)\ne (\bar 0,\bar 0)$, the kernel of the action of $\frt_\mu$ is $\oplus_{i=0}^2\iota_i\bigl(k\otimes(\calO_{(\bar\jmath_1,\bar\jmath_2)}\oplus
\calO_{(-\bar\jmath_1,-\bar\jmath_2)})\bigr)$.
\end{corollary}

\begin{proof}
Note that $\frt_{(\bar 0,\bar 0,\bar 1)}\oplus\frt_{(\bar 0,\bar 0,\bar 2)}$ is the natural Cartan subalgebra of $\frt\cong \fro(\calO,n)$ in terms of the basis of $\calO$ in Table \ref{ta:Oalphabeta}. The Lemma above shows that the one-dimensional homogeneous subspaces $\frt_{(\bar\jmath_1,\bar\jmath_2,\bar\jmath_3)}$ also act semisimply for any $(\bar\jmath_1,\bar\jmath_2)\ne (\bar 0,\bar 0)$, as well as the assertion on the kernel.
\end{proof}

\medskip

Actually, the $\bZ_3^3$-grading of type $(0,26)$ on $\frg(k,\calO)$ in Proposition \ref{pr:gradingsF4}\,(2) can be extended to a $\bZ_3^3$-grading on $\frg=\frg(S_2,\calO)$ of type $E_6$, where $S_2$ is a two-dimensional symmetric composition algebra with basis $\{a,b\}$, where $n(a)=0=n(b)$, $a\bullet a=b$, $b\bullet b=\xi a$, with $\xi=n(a,a\bullet a)$, and $a\bullet b=b\bullet a=0$. The triality Lie algebra of $S_2$ is
\[
\frs=\{\bigl(\epsilon\sigma_{a,b},\delta\sigma_{a,b},\gamma\sigma_{a,b}\bigr) : \epsilon,\delta,\gamma\in k,\ \epsilon+\delta+\gamma=0\}
\]
(see \cite[Corollary 3.4]{Eld04Ibero}). Let $\sigma=\sigma_{a,b}$, the order $3$ automorphism $\theta$ grades $\frs$ over $\bZ_3$ with $\frs_{\bar 0}=0$, $\frs_{\bar 1}=\espan{(\sigma,\omega^2\sigma,\omega\sigma)}$ and $\frs_{\bar 2}=\espan{(\sigma,\omega\sigma,\omega^2\sigma)}$ (with $\omega^3=1\ne \omega$ as usual). Note that $\sigma$ acts semisimply on $S$, so that $\frs$ is a two-dimensional abelian toral subalgebra of $\frg$, that is, its elements act semisimply on $\frg$.

The homogeneous subspaces of the $\bZ_3^3$-grading on the Lie algebra $\frg=\frg(S_2,\calO)= (\frs\oplus\frt)\oplus\bigl(\oplus_{i=0}^2\iota_i(S_2\otimes\calO)\bigr)$ induced by the standard $\bZ_3^2$-grading of $\calO$ and the order $3$ automorphism $\Theta$ are the following:
\[
\begin{split}
\frg_{(\bar 0,\bar 0,\bar\jmath)}&=\frs_{\bar\jmath}\oplus \frt_{(\bar 0,\bar 0,\bar\jmath)},\quad\text{for $j=1,2$,}\\[6pt]
\frg_{(\bar\jmath_1,\bar\jmath_2,\bar\jmath_3)}&=
 \frt_{(\bar\jmath_1,\bar\jmath_2,\bar\jmath_3)}\\
 &\qquad \oplus\espan{\iota(u\otimes x)+\omega^{2j_3}\iota_1(u\otimes x)+\omega^{j_3}\iota_2(u\otimes x): u\in S_2, x\in\calO_{(\bar\jmath_1,\bar\jmath_2)}},\\
 &\qquad\qquad\qquad \text{for $(\bar\jmath_1,\bar\jmath_2)\ne(\bar 0,\bar0)$.}
\end{split}
\]
For any $u\in S_2$, $x\in\calO$, and $j=0,1,2$, denote by $\Gamma_j(u\otimes x)$ the element $\iota_0(u\otimes x)+\omega^{2j}\iota_1(u\otimes x)+\omega^j\iota_2(u\otimes x)$, and for any $D=(d_0,d_1,d_2)\in \frt$, denote by $T_j(D)$ the triple $D+\omega^{2j}\theta(D)+\omega^j\theta^2(D)$ (and similarly for $D\in \tri(S_2)=\frs$).

For $(\bar\jmath_1,\bar\jmath_2)\ne (\bar 0,\bar 0)$, $0\ne u\in S_2$ and $0\ne x\in \calO_{(\bar\jmath_1,\bar\jmath_2)}$, the element $\Gamma_j(u\otimes x)$ is a nonzero homogeneous element in $\frg_{(\bar\jmath_1,\bar\jmath_2,\bar\jmath)}$. Also, for any nonzero homogeneous elements $x\in \calO_{(\bar\imath_1,\bar\imath_2)}$ and $y\in\calO_{(\bar\jmath_1,\bar\jmath_2)}$, the element $T_j(t_{x,y})$ is a nonzero homogeneous element in $\frt_{(\bar\imath_1+\bar\jmath_1,\bar\imath_2+\bar\jmath_2,\bar\jmath)}$.

Note that for $0\ne u,v\in S_2$ and $0\ne x,y\in \calO$, and $i,j=0,1,2$:
\[
\begin{split}
[\Gamma_i&(u\otimes x),\Gamma_j(v\otimes y)]\\
 &=[\iota_0(u\otimes x)+\omega^{2i}\iota_1(u\otimes x)
  +\omega^i\iota_2(u\otimes x),
  \iota_0(v\otimes y)+\omega^{2j}\iota_1(v\otimes y)+
  \omega^j\iota_2(v\otimes y)]\\
 &=n(x,y)\bigl(t_{u,v}+\omega^{2(i+j)}\theta(t_{u,v})+\omega^{i+j}\theta^2(t_{u,v})\bigr)\\
  &\quad+n(u,v)\bigl(t_{x,y}+\omega^{2(i+j)}\theta(t_{x,y})+\omega^{i+j}\theta^2(t_{x,y})\bigr)\\
  &\quad+\omega^{2i+j}\iota_0(u\bullet v\otimes x*y)-\omega^{i+2j}\iota_0(v\bullet u\otimes y*x)\\
  &\quad+\omega_i\iota_1(u\bullet v\otimes x*y)-\omega^j\iota_1(v\bullet u\otimes y*x)\\
  &\quad+\omega^{2j}\iota_2(u\bullet v\otimes x*y)-\omega^{2i}\iota_2(v\bullet u\otimes y*x)\\[2pt]
 &=n(x,y)T_{i+j}(t_{u,v})+n(u,v)T_{i+j}(t_{x,y})
   +\Gamma_{i+j}(u\bullet v\otimes (\omega^{2i+j}x*y-\omega^{i+2j}y*x)),
\end{split}
\]
since $(S_2,\bullet)$ is commutative.

We want to check that for any $0\ne \mu\in \bZ_3^3$
\begin{center}
$\frg_\mu\oplus\frg_{-\mu}$ is a Cartan subalgebra of $\frg$.
\end{center}

This is clear for $\mu=\pm(\bar 0,\bar 0,\bar 1)$, as $\frg_\mu\oplus\frg_{-\mu}=\frs\oplus \frt_{(\bar 0,\bar 0,\bar 1)}\oplus \frt_{(\bar 0,\bar 0,\bar 2)}$ is the natural Cartan subalgebra of $\frs\oplus\frt\simeq \frs\oplus\fro(\calO,n)$. Now, consider the element $\mu=(\bar\imath_1,\bar\imath_2,\bar\imath)$ with $(\bar\imath_1,\bar\imath_2)\ne (\bar 0,\bar 0)$. Let $\hat\mu=(\bar\imath_1,\bar\imath_2)$ and $0\ne x\in \calO_{\hat \mu}$, so $0\ne x*x\in \calO_{-\hat\mu}$. Then,
\[
\frg_\mu\oplus\frg_{-\mu}=\bigl(\frt_\mu\oplus \Gamma_i(S_2\otimes x)\bigr)\oplus \bigl(\frt_{-\mu}\oplus \Gamma_{-i}(S_2\otimes (x*x))\bigr).
\]
But $\frt_\mu$ and $\frt_{\mu}$ annihilate $\oplus_{i=0}^2\iota_i\bigl(S_2\otimes(\calO_{\hat\mu}\oplus\calO_{-\hat\mu})\bigr)$ and, in particular, they annihilate $\Gamma_{\bar\imath}(S_2\otimes x)$ and $\Gamma_{-\bar\imath}(S_2\otimes (x*x))$. This shows that $[\frg_\mu,\frg_\mu]=0=[\frg_{-\mu},\frg_{-\mu}]$. Since $[\frg_\mu,\frg_{-\mu}]\subseteq \frg_0=0$, it follows that $\frg_\mu\oplus\frg_{-\mu}$ is an abelian subalgebra of $\frg$.

On the other hand, $\frt_\mu$ and $\frt_{-\mu}$ act semisimply on $\oplus_{i=0^2}\iota_i(S_2\otimes\calO)$ by Corollary \ref{co:tmutoral}, and this subspace generates $\frg$, so they act semisimply on $\frg$. Also, for $\hat\nu=(\bar\jmath_1,\bar\jmath_2)\ne(\bar 0,\bar 0)\ne \pm\hat\mu$ and $0\ne y\in \calO_{\hat\nu}$, either $x*y=0$ or $y*x=0$, but not both (Theorem \ref{th:bonito}), and
\[
[\Gamma_i(a\otimes x),\Gamma_j(a\otimes y)]=\Gamma_{i+j}\bigl(b\otimes(\omega^{2i+j}x*y-\omega^{i+2j}y*x)\bigr)\ne 0,
\]
\[
\begin{split}
[\Gamma_i(a\otimes x),
 &\Gamma_{i+j}(b\otimes(\omega^{2i+j}x*y-\omega^{i+2j}y*x))] \\
 &=T_{2i+j}(t_{x,\omega^{2i+j}x*y-\omega^{i+2j}y*x}\in T_{2i+j}(t_{\calO_{\hat\mu,},\calO_{\hat\nu+\hat\nu}}),
\end{split}
\]
and
\[
0\ne [T_{2i+j}(t_{\calO_{\hat\mu},\calO_{\hat\mu+\hat\nu}}),\Gamma_i(a\otimes x)]\subseteq \Gamma_j(a\otimes\calO_{\hat\nu})=k\Gamma_j(a\otimes y)\subseteq \frg_{(\bar\jmath_1,\bar\jmath_2,\bar\jmath)},
\]
since Lemma \ref{le:toral} shows that $[T_{2i+j}(D),\Gamma_i(a\otimes x)]$ is not $0$ for $D$ any endomorphism of the form $\epsilon\sigma_{x,z}+\delta(\frac{1}{2}n(x,z)1-r_xl_z)+\gamma(\frac{1}{2}n(x,z)1-l_xr_z)$ for $0\ne z\in \calO_{\hat\mu+\hat\nu}$. Hence $\ad_{\Gamma_i(a\otimes x)}^3(\Gamma_j(a\otimes y))$ is a nonzero scalar multiple of $\Gamma_j(a\otimes y)$ and, therefore, $\ad_{\Gamma_i(a\otimes x)}$ acts in a semisimple way on $\Gamma_j(a\otimes \calO_{\hat\nu})\oplus\Gamma_{i+j}(b\otimes \calO_{\hat\mu+\hat\nu})\oplus T_{2i+j}(t_{\calO_{\hat\mu},\calO_{\hat\mu+\hat\nu}})$. Hence $\ad_{\Gamma_i(a\otimes x)}$ acts semisimply on the subspace
\[
\Bigl(\sum_{j=0}^2\sum_{\hat\nu\ne\pm\hat\mu}\Gamma_j(S_2\otimes \calO_{\hat\nu}\Bigr)
\oplus\Bigl(\sum_{j=0}^2\sum_{\hat\nu\ne\pm\hat\mu}
 T_j\bigl(t_{\calO_{\hat\mu},\calO_{\hat\nu}}\bigr)\Bigr).
\]
Since $\sum_{\nu\ne\pm\mu}\Bigl(\iota_0(S_2\otimes \calO_{\hat\nu})
\oplus\iota_1(S_2\otimes \calO_{\hat\nu})\oplus\iota_2(S_2\otimes \calO_{\hat\nu})\Bigr)$ is contained in this subspace, and it generates the whole Lie algebra $\frg$, it follows that $\ad_{\Gamma_i(a\otimes x)}$ acts semisimply on $\frg$.

The same argument works with $\ad_{\Gamma_i(b\otimes x)}$, $\ad_{\Gamma_i(a\otimes x*x)}$ and $\ad_{\Gamma_i(b\otimes x*x)}$. Therefore, $\frg_\mu\oplus\frg_{-\mu}$ is an abelian toral subalgebra of $\frg$.

Note that in characteristic $0$, the fact that $\frg_0=0$ already implies that $\frg_\mu$ consists of semisimple elements (\cite[Chapter 3, Corollary to Theorem 3.4]{LGLAIII}), so the arguments above are not necessary.

The next result summarizes the previous work:

\begin{theorem}\label{th:toral}
The $\bZ_3^3$-grading of type $(0,26)$ of the simple Lie algebra $\frg(k,\calO)$ of type $F_4$ and the $\bZ_3^3$-grading of type $(0,0,26)$ of the simple Lie algebra $\frg(S_2,\calO)$ of type $E_6$ satisfy that $\frg_0=0$ and that $\frg_\mu\oplus\frg_{-\mu}$ is a Cartan subalgebra for any $0\ne \mu\in \bZ_3^3$.
\end{theorem}

\begin{proof}
For $\frg(S_2,\calO)$ it follows from the previous arguments. The result for $\frg(k,\calO)$ follows by restriction, since the ground field $k$ is a subalgebra of the two-dimensional symmetric composition algebra $S_2$ as above with $\xi=1$.
\end{proof}

\medskip

\subsection{Gradings on $E_8$}

Let $(S,*,n)$ and $(S',\star,n')$ be two eight dimensional symmetric composition algebras and let $\frg=\frg(S,S')$ be the Lie algebra of type $E_8$ constructed in \eqref{eq:gSS'}.

By considering different possibilities for gradings on $S$ and $S'$ and combining these gradings with either the natural $\bZ_2^2$-grading of $\frg(S,S')$ or the $\bZ_3$-grading induced on $\frg(S,S')$ by the order $3$ automorphism $\Theta$ in \eqref{eq:Theta} (assuming the ground field contains the cubic roots of $1$), there appears a bunch of gradings on the Lie algebra $\frg(S,S')$.

Thus, for instance, by combining a $\bZ_3$-grading on a two-dimensional symmetric composition algebra $S_2$, a $\bZ_2^3$-grading on a para-Cayley algebra $S_8$, and the $\bZ_3$-grading induced by the automorphism $\Theta$ in \eqref{eq:Theta}, one gets an interesting $\bZ_2^3\times \bZ_3^2$-grading on the exceptional simple Lie algebra $\frg(S_2,S_8)$ of type $E_6$.

Many different gradings like this one can be obtained for the exceptional simple Lie algebras. We will sketch here two of these gradings for the simple Lie algebra of type $E_8$.

\medskip

Assume first that both $(S,*,n)$ and $(S',\star,n')$ are $\bZ_2^3$-graded para-Cayley algebras. Combine these gradings with the natural $\bZ_2^2$-grading on $\frg=\frg(S,S')$ to obtain a $\bZ_2^3\times\bZ_2^3\times\bZ_2^2=\bZ_2^8$-grading on $\frg$. Using our results on gradings on the orthogonal Lie algebras $\fro(S,n)$ and $\fro(S',n')$, we obtain:
\[
\begin{split}
\frg_{(\mu,0,(\bar 0,\bar 0))}&=\tri(S,*,n)_\mu\\
 &\qquad\qquad\text{(a Cartan subalgebra of $\tri(S,*,n)\simeq\fro(S,n)$ if $\mu\ne 0$),}\\
\frg_{(0,\nu,(\bar 0,\bar 0))}&=\tri(S',\star,n)_\nu\\
 &\qquad\qquad\text{(a Cartan subalgebra of $\tri(S',\star,n')\simeq\fro(S',n')$, if $\nu\ne 0$),}\\
\frg_{(\mu,\nu,(\bar 1,\bar 0))}&=\iota_0(S_\mu\otimes S_\nu),\\
\frg_{(\mu,\nu,(\bar 0,\bar 1))}&=\iota_1(S_\mu\otimes S_\nu),\\
\frg_{(\mu,\nu,(\bar 1,\bar 1))}&=\iota_2(S_\mu\otimes S_\nu),
\end{split}
\]
for any $\mu,\nu\in\bZ_2^3$, thus getting a grading of type $(192,0,0,14)$. Note that $\frg_0=0$.

\medskip

On the other hand, if both $S$ and $S'$ are Okubo algebras endowed with standard $\bZ_3^2$-gradings over a field containing the cubic roots of $1$, then $\frg=\frg(S,S')$ is naturally endowed with a $\bZ_3^2\times\bZ_3^2\times\bZ_3=\bZ_3^5$-grading where, for any $0\ne \mu,\nu\in\bZ_3^2$:
\[
\begin{split}
\frg_{(\mu,0,\bar \jmath)}&=\{(d_0,d_1,d_2)\in\tri(S,*,n)_\mu: \theta((d_0,d_1,d_2))=\omega^j(d_0,d_1,d_2)\},\\
\frg_{(0,\nu,\bar \jmath)}&=\{(d_0',d_1',d_2')\in\tri(S',\star,n')_\nu: \theta'((d_0',d_1',d_2'))=\omega^j(d_0',d_1',d_2')\},\\
\frg_{(0,0,\bar \jmath)}&=\{(d_0,d_1,d_2)\in\tri(S,*,n)_0: \theta((d_0,d_1,d_2))=\omega^j(d_0,d_1,d_2)\}\\
&\qquad\oplus \{(d_0',d_1',d_2')\in\tri(S',\star,n')_0: \theta'((d_0',d_1',d_2'))=\omega^j(d_0',d_1',d_2')\},\\
\frg_{(\mu,\nu,\bar 0)}&=k\bigl(\iota_0(a_\mu\otimes b_\nu)+
\iota_1(a_\mu\otimes b_\nu)+\iota_2(a_\mu\otimes b_\nu),\\
\frg_{(\mu,\nu,\bar 1)}&=k\bigl(\iota_0(a_\mu\otimes b_\nu)+
\omega^2\iota_1(a_\mu\otimes b_\nu)+\omega\iota_2(a_\mu\otimes b_\nu),\\
\frg_{(\mu,\nu,\bar 2)}&=k\bigl(\iota_0(a_\mu\otimes b_\nu)+
\omega\iota_1(a_\mu\otimes b_\nu)+\omega^2\iota_2(a_\mu\otimes b_\nu)\bigr),
\end{split}
\]
where $S_\mu=ka_\mu$ and $S_\nu'=kb_\nu$.

Note that again $\frg_0=0$. The type of this $\bZ_3^5$-grading is then $(240,0,0,2)$.

\medskip

Let us summarize these arguments:

\begin{proposition}
Let $k$ be a field of characteristic $\ne 2,3$ containing the cubic roots of $1$.
\begin{enumerate}
\item If $\bar C$ and $\bar C'$ are two $\bZ_2^3$-graded para-Cayley algebras, these gradings induce a $\bZ_2^8$-grading on the simple Lie algebra $\frg(\bar C,\bar C')$ of type $(192,0,0,14)$.

\item If $\calO$ and $\calO'$ are two $\bZ_3^2$-graded Okubo algebra, these gradings induce a $\bZ_3^5$-grading on the simple Lie algebra $\frg(\calO,\calO')$ of type $(240,0,0,2)$.
\end{enumerate}
\end{proposition}

\bigskip

Consider again the case in which both $(S,*,n)$ and $(S',\star,n')$ are $\bZ_2^3$-graded para-Cayley algebras. The projection
\[
\begin{split}
\bZ_2^8=\bZ_2^3\times\bZ_2^3\times\bZ_2^2&\longrightarrow \bZ_2^3\times\bZ_2^2=\bZ_2^5\\
(\mu,\nu,\gamma)&\mapsto (\mu+\nu,\gamma),
\end{split}
\]
provides a coarsening of the previous $\bZ_2^8$-grading of $\frg=\frg(S,S')$ to a $\bZ_2^5$-grading.

Here again $\frg_0=0$, and for any $\mu\in\bZ_2^3$:
\begin{equation}\label{eq:e8mu00}
\frg_{(\mu,(\bar 0,\bar 0))}=\tri(S,*,n)_\mu\oplus\tri(S',\star,n')_\mu,
\end{equation}
which is a Cartan subalgebra of $\tri(S,*,n)\oplus\tri(S',\star,n')$, and hence of the whole Lie algebra $\frg$. On the other hand, we have:
\begin{equation}\label{eq:e8muij}
\begin{split}
\frg_{(\mu,(\bar 1,\bar 0))}&=\oplus_{\nu\in\bZ_2^3}\iota_0(S_\nu\otimes S_{\mu+\nu}'),\\
\frg_{(\mu,(\bar 0,\bar 1))}&=\oplus_{\nu\in\bZ_2^3}\iota_1(S_\nu\otimes S_{\mu+\nu}'),\\
\frg_{(\mu,(\bar 1,\bar 1))}&=\oplus_{\nu\in\bZ_2^3}\iota_2(S_\nu\otimes S_{\mu+\nu}').
\end{split}
\end{equation}
All these subspaces are eight dimensional abelian subalgebras of $\frg$ (as $\frg_0=0$). Besides, for any $\mu,\nu\in\bZ_2^3$, $0\ne a\in S_\mu$, $0\ne x\in S_\nu'$, and $i=0,1,2$, let us show that the adjoint map $\ad_{\iota_i(a\otimes x)}$ is a semisimple endomorphism. First note that
\[
[\iota_i(a\otimes x),\iota_i(S_{\hat\mu}\otimes S_{\hat\nu}']=0
\]
if either $\hat\mu\ne \mu$ and $\hat\nu\ne\nu$ or $(\hat\mu,\hat\nu)=(\mu,\nu)$, while for $y\in S_{\hat\nu}'$, $\hat\nu\ne \nu$:
\[
\begin{split}
[\iota_i(a\otimes x),[\iota_i(a\otimes x)&,\iota_i(a\otimes y)]]\\
 &=[\iota_i(a\otimes x),2n(a)\theta'^i(t_{x,y}')]\\
 &=-2n(a)\iota_i(a\otimes\sigma_{x,y}(x))=-4n(a)n'(x)\iota_i(a\otimes y).
\end{split}
\]
Note that $n(a)\ne 0\ne n'(x)$ because all the homogeneous spaces in these $\bZ_2^3$-gradings are nonisotropic.

Similarly,
\[
\ad_{\iota_i(a\otimes x)}^2(\iota_i(b\otimes x))=-4n(a)n'(x)\iota_i(b\otimes x)
\]
for $b\in S_{\hat\mu}$, $\hat\mu\ne\mu$.

Also, for $(d_0,d_1,d_2)\in \tri(S,*,n)$:
\[
\begin{split}
\ad_{\iota_i(a\otimes x)}((d_0,d_1,d_2))&=-\iota_i(d_i(a)\otimes x),\\[2pt]
\ad_{\iota_i(a\otimes x)}^2((d_0,d_1,d_2))
 &=-[\iota_i(a\otimes x),\iota_i(d_i(a)\otimes x)]\\
 &=-2n'(x)\theta^i(t_{a,d_i(a)})\quad\text{as $n(a,d_i(a))=0$,}\\[2pt]
\ad_{\iota_i(a\otimes x)}^3((d_0,d_1,d_2))
 &=2n'(x)\iota_i(\sigma_{a,d_i(a)}(a)\otimes x)\\
 &=4n(a)n'(x)\iota_i(d_i(a)\otimes x),
\end{split}
\]
so $\ad_{\iota_i(a\otimes x)}^3=-4n(a)n'(x)\ad_{\iota_i(a\otimes x)}$ on $\iota_i(S\otimes S')$ and on $\tri(S,*,n)$, and with the same arguments this works too on $\tri(S',\star,n')$.

On the other hand, for any $b\in S$ and $y\in S'$:
\[
\begin{split}
\ad_{\iota_i(a\otimes x)}^2(\iota_{i+1}(b\otimes y))
 &=[\iota_i(a\otimes x),\iota_{i+2}(a*b\otimes x\star y)]\\
 &=-\iota_{i+1}\bigl((a*b)*a\otimes (x*y)*x\bigr)\\
 &=-n(a)n'(x)\iota_{i+1}(b\otimes y),
\end{split}
\]
and thus the restriction of $\ad_{\iota_i(a\otimes x)}^2$ to $\iota_{i+1}(S\otimes S')\oplus\iota_{i+2}(S\otimes S')$ is $-n(a)n'(x)$ times the identity.

The conclusion is that the eight dimensional abelian subalgebras $\frg_{(\mu,(\bar 1,\bar 0))}$, $\frg_{(\mu,(\bar 0,\bar 1))}$ and $\frg_{(\mu,(\bar 1,\bar 1))}$, for $\mu\in\bZ_2^3$, are all toral subalgebras, and hence Cartan subalgebras of $\frg=\frg(S,S')$. Therefore:

\begin{proposition}\label{pr:E8Dempwolff}
The $\bZ_2^5$-grading on the simple Lie algebra $\frg(S,S')$ of type $E_8$ given by \eqref{eq:e8mu00} and \eqref{eq:e8muij} is a Dempwolff decomposition.
\end{proposition}

\smallskip

Thus, Dempwolff decompositions appear naturally related to the $\bZ_2^3$-gradings on Cayley algebras.

It must be remarked here that Thompson proved in \cite{Thompson} that the automorphism group of the simple complex Lie algebra of type $E_8$ acts transitively on Dempwolff decompositions.

A final comment is in order here:

\begin{remark}\label{re:Jordangradings}
As mentioned in \cite[Chapter 3. \S 3.13]{LGLAIII}, Alekseevskij \cite{Alek} classified all the \emph{Jordan gradings} on the exceptional complex simple Lie algebras. These are gradings in which the zero homogeneous space is trivial, and all the other homogeneous spaces have the same dimension and consist of semisimple elements. The grading over $\bZ_2^3$ of type $(0,7)$ of the simple Lie algebra of type $G_2$ in the paragraph previous to Proposition \ref{pr:Z23D4}, the $\bZ_3^3$-gradings of types $(0,26)$ and $(0,0,26)$ of the simple Lie algebras of types $F_4$ and $E_6$ respectively in Theorem \ref{th:toral}, as well as the Dempwolff decomposition of $E_8$ in Proposition \ref{pr:E8Dempwolff} exhaust these Jordan gradings with the exception of a $\bZ_5^3$-grading of $E_8$, in which all homogeneous spaces have dimension $2$. This is the only Jordan grading that seems not to be related to gradings on composition algebras.
\end{remark}



\def\cprime{$'$}
\providecommand{\bysame}{\leavevmode\hbox to3em{\hrulefill}\thinspace}
\providecommand{\MR}{\relax\ifhmode\unskip\space\fi MR }

\end{document}